\documentclass{amsart}%

\usepackage{amssymb}
\usepackage{amsmath}
\usepackage{amsthm}
\usepackage{amscd,mdwlist}
\usepackage[margin=1in]{geometry}
\usepackage{hyperref}
\usepackage[usenames,dvipsnames]{xcolor}
\usepackage{cite}
\usepackage[utf8]{inputenc}
\usepackage{csquotes, esint}    
    
\theoremstyle{plain}
\newtheorem{theorem}{Theorem}[section]

\newtheorem{corollary}[theorem]{Corollary}
\newtheorem{lemma}[theorem]{Lemma}

\theoremstyle{remark}
\newtheorem{remark}[theorem]{Remark}
\newtheorem{examples}[theorem]{Example}

\newcommand{\nocontentsline}[3]{}
\let\origcontentsline\addcontentsline
\newcommand\stoptoc{\let\addcontentsline\nocontentsline}
\newcommand\resumetoc{\let\addcontentsline\origcontentsline}

\DeclareMathOperator{\supp}{supp}
\DeclareMathOperator{\diam}{diam}

\DeclareMathOperator{\cpct}{Cap}
\DeclareMathOperator{\dcpct}{\dot{C}ap}

\newcommand{\R}{\mathbb{R}}
\newcommand{\cH}{\mathcal{H}}

\makeatletter
\@namedef{subjclassname@2020}{\textup{2020} Mathematics Subject Classification}
\makeatother

\begin{document}

\title{Kernels of trace operators via fine continuity}
\author{M. Hinz$^1$, S.~N. Chandler-Wilde$^2$, D.~P. Hewett$^3$ }
\address{$^1$ Department of Mathematics,
Bielefeld University, 33501 Bielefeld, Germany}
\email{mhinz@math.uni-bielefeld.de}
\address{$^2$ 
Department of Mathematics and Statistics,
University of Reading,
Whiteknights,
PO Box 220, Reading, RG6 6AX, United Kingdom}
\email{s.n.chandler-wilde@reading.ac.uk}
\address{$^3$ 
Department of Mathematics,
University College London,
Gower Street,
London WC1E 6BT, United Kingdom}
\email{d.hewett@ucl.ac.uk}
\thanks{$^1$ Partially supported by DFG CRC 1283, \enquote{Taming uncertainty and profiting from randomness and low regularity in analysis, stochastics and their applications}.}
\thanks{$^2$ Support from the UK Spectral Theory Network, via the Isaac Newton Institute, Cambridge and EPSRC grant EP/V521929/1 gratefully acknowledged.}

\date{\today}
\begin{abstract} 
Given a closed subset $\Gamma$ of $\mathbb{R}^n$ 
that is the support of a measure $\mu$, we study the kernels  of  trace operators from fractional Sobolev spaces $H_p^\alpha(\mathbb{R}^n)$ into the space of $\mu$-equivalence
classes of functions on $\Gamma$. We characterise these kernels as the closure of $C_c^\infty(\mathbb{R}^n\setminus \Gamma)$
in $H_p^\alpha(\mathbb{R}^n)$, provided quasi continuous representatives of elements of $H_p^\alpha(\mathbb{R}^n)$ have the following key property: they
vanish quasi everywhere on $\Gamma$ if and only if they vanish $\mu$-almost 
everywhere on $\Gamma$. We establish that this key property holds if the measures satisfy localized upper density conditions. Such measures need not be doubling, in particular the set $\Gamma$ may be a finite union of closed sets having different 
Hausdorff dimensions. We provide corresponding results for spaces 
$H_p^\alpha(\Omega)$ on domains $\Omega\subset \mathbb{R}^n$ satisfying a weakened version of the measure density condition. We observe that the above key property is essential for the convergence of Galerkin integral equation methods, based on integration with respect to the measure $\mu$, for certain BVPs in the complement of $\Gamma$.

\tableofcontents
\end{abstract}


\maketitle

\section{Introduction}
\label{S:intro}

The study of trace operators, and the classification of their images and kernels, is a central topic in the theory of function spaces. It also plays a crucial role in the mathematical analysis of 
PDEs and integral equations. We focus here on the case of fractional Sobolev spaces (Bessel potential spaces) $H^\alpha_p(\R^n)$, with $1<p<\infty$ and $\alpha>0$, and trace spaces on closed subsets $\Gamma\subset\R^n$.
Given $\Gamma$,
one seeks to define a bounded linear operator ${\mathrm T}:H^\alpha_p(\R^n) \to B(\Gamma)$ mapping $H^\alpha_p(\R^n)$ surjectively onto a suitable Banach space $B(\Gamma)$ of functions defined on $\Gamma$, with the property that if $u\in H^\alpha_p(\R^n)$ is sufficiently smooth, then $\mathrm{T}u$ corresponds to pointwise restriction, $\mathrm{T} u=u|_\Gamma$, at least outside some negligible set. A fundamental task 
is then to characterize the kernel $\ker \mathrm{T}$ of the operator $\mathrm{T}$ as a subspace of $H^\alpha_p(\R^n)$, noting that $B(\Gamma)$ is isomorphic to the quotient of $H^\alpha_p(\R^n)$ by this kernel. 
If $\Gamma$ is a smooth submanifold of $\mathbb{R}^n$, the boundary of a  Lipschitz domain, or more generally a $d$-set (i.e., the support of an Ahlfors $d$-regular measure) \cite{JW84, Triebel97} for some $0<d<n$,
these questions are 
well understood. 
However, for less regular $\Gamma$ (e.g., when the local fractal dimension varies across $\Gamma$), while definitions of  ${\mathrm T}:H^\alpha_p(\R^n) \to B(\Gamma)$ and characterizations of $B(\Gamma)$ have been established for quite general $\Gamma$, there remain fundamental open questions
concerning the characterization of $\ker \mathrm{T}$, which we address in this paper.  
 
One very general approach to defining trace operators is to take a purely potential-theoretic point of view. It is well known that each element $u\in H^\alpha_p(\mathbb{R}^n)$ has an $(\alpha,p)$-quasi continuous representative $\widetilde{u}$, uniquely determined $(\alpha,p)$-q.e. Take $B(\Gamma)$ to be the set $\widetilde{B}(\Gamma)$ of equivalence classes with respect to $(\alpha,p)$-q.e.\ coincidence of restrictions $\widetilde{u}|_\Gamma$ 
for all such $\widetilde{u}$. We can then define the trace $\mathrm{Tr_{\Gamma,0}}u$ of $u\in H^\alpha_p(\mathbb{R}^n)$ to be the equivalence class of $\widetilde{u}|_\Gamma$. This gives a surjective linear map
\begin{equation}\label{E:traceopqe}
\mathrm{Tr_{\Gamma,0}}:H^\alpha_p(\mathbb{R}^n)\to \widetilde{B}(\Gamma)
\end{equation}
whose kernel $\ker \mathrm{Tr_{\Gamma,0}}$ is the set of all $u\in H^\alpha_p(\R^n)$ for which $\widetilde{u}=0$ $(\alpha,p)$-q.e.\ on $\Gamma$. Standard arguments, \cite[Proposition 2.3.8]{AH96}, show that $\ker \mathrm{Tr_{\Gamma,0}}$ is closed in $H^\alpha_p(\mathbb{R}^n)$. Therefore $\widetilde{B}(\Gamma)$, endowed with the quotient norm, is a Banach space. In this situation $\mathrm{Tr_{\Gamma,0}}$ is a bounded linear operator with norm one.

A well-known theorem of Netrusov \cite{Netrusov93} (see also \cite[Corollary 10.1.2]{AH96})
implies that if 
\begin{equation}\label{E:Netrusov0}
0<\alpha<\alpha_p(\Gamma)+1\quad \text{or}\quad\text{ $\alpha=\alpha_p(\Gamma)+1$ and the $(\alpha_p(\Gamma),p)$-capacity of $\Gamma$ is zero},
\end{equation}
then 
\begin{align}
\label{E:kerTr}
\ker \mathrm{Tr_{\Gamma,0}}=\widetilde{H}^\alpha_p(\mathbb{R}^n\setminus \Gamma),
\end{align}  
where $\widetilde{H}^\alpha_p(\mathbb{R}^n\setminus \Gamma)$ denotes the closure in $H^\alpha_p(\R^n)$ of the set $C^\infty_c(\mathbb{R}^n\setminus\Gamma)$ of all smooth functions compactly supported in $\mathbb{R}^n\setminus \Gamma$ and
\begin{equation}\label{E:capdim}
\alpha_p(\Gamma):=\frac{1}{p}(n-\dim_H\Gamma),
\end{equation}
$\dim_H\Gamma$ being the Hausdorff dimension \cite{Falconer97, Mattila, Zaehle} of $\Gamma$. (This kind of result is sometimes described as \enquote{spectral synthesis}, see for instance \cite{Beurling49, Hedberg81, HedbergNetrusov07} and \cite[Section 9.13]{AH96}.) The Netrusov theorem gives a similar result for larger $\alpha$ if one considers higher-order trace operators involving also traces of partial derivatives. Earlier results for special cases were shown by Sobolev \cite{Sobolev63}, Deny \cite{Deny50}, Havin \cite{Havin68}, Bagby \cite{Bagby72} and Hedberg \cite{Hedberg72}. For details see Section 3 below, or \cite[Sections 9.1 and 10.1]{AH96}. 

On the other hand, in applications one often works with a closed set $\Gamma\subset\mathbb{R}^n$ that is the support $\Gamma=\supp \mu$ of some Radon measure $\mu$, which models a specific distribution of mass across $\Gamma$. A classical example is where $\Gamma=\partial\Omega$ is the boundary of a Lipschitz domain $\Omega$, with $\mu$ the surface measure on $\Gamma$. In this context it is desirable that the image space $B(\Gamma)$ be a subspace of $L^p(\Gamma,\mu)$. This is possible provided $\mu$ and hence, implicitly, $\Gamma$ satisfies appropriate conditions, see for instance \cite[Theorems 7.2.1 and 7.2.2]{AH96}. In particular, $\mu$ should be absolutely continuous with respect to the $(\alpha,p)$-capacity, i.e., $\mu(E)=0$ for every Borel set $E\subset \mathbb{R}^n$ of zero $(\alpha,p)$-capacity. In this absolutely continuous case, we can consider the space $[\widetilde{B}(\Gamma)]_\mu$ of equivalence classes with respect to $\mu$-a.e.\ equality of restrictions $\widetilde{u}|_\Gamma$, 
and define the trace $[\mathrm{Tr_{\Gamma,0}}]_\mu u$ of $u\in H^\alpha_p(\mathbb{R}^n)$ to be the equivalence class of $\widetilde{u}|_\Gamma$ in this sense. Then
\begin{equation}\label{E:traceopmu}
[\mathrm{Tr_{\Gamma,0}}]_\mu:H^\alpha_p(\mathbb{R}^n)\to [\widetilde{B}(\Gamma)]_\mu 
\end{equation}
is a surjective linear map whose kernel $\ker [\mathrm{Tr_{\Gamma,0}}]_\mu$ is the set of all $u\in H^\alpha_p(\R^n)$ for which $\widetilde{u}=0$ $\mu$-a.e.\ on $\Gamma$. Moreover, $\ker [\mathrm{Tr_{\Gamma,0}}]_\mu$ is closed in $H^\alpha_p(\mathbb{R}^n)$, so that $[\widetilde{B}(\Gamma)]_\mu$, endowed with the quotient norm, is a Banach space. With respect to this quotient norm, the linear operator $[\mathrm{Tr_{\Gamma,0}}]_\mu$ is bounded with operator norm one.


The main 
goal of 
this paper is to provide a characterisation of the kernel $\ker [\mathrm{Tr_{\Gamma,0}}]_\mu$ of $[\mathrm{Tr_{\Gamma,0}}]_\mu$. 
Specifically, our aim is to determine general sufficient conditions
under which
\begin{align}
\label{E:kerTrmu}
\ker [\mathrm{Tr_{\Gamma,0}}]_\mu=
 \widetilde{H}^\alpha_p(\mathbb{R}^n\setminus \Gamma).
\end{align}
 We also provide similar results relating to higher-order trace operators involving traces of partial derivatives.

It is evident that if \eqref{E:kerTr} holds then 
\eqref{E:kerTrmu} is equivalent to the requirement that
\begin{equation}\label{E:wish}
\text{$\widetilde{u}=0$ $\mu$-a.e.\ on $\Gamma$}\quad\text{if and only if}\quad\text{$\widetilde{u}=0$ $(\alpha,p)$-q.e.\ on $\Gamma$,} \qquad u\in H^\alpha_p(\R^n).
\end{equation}
Thus 
our focus in this paper will be on determining conditions under which \eqref{E:wish} holds. 
We emphasize that 
while one implication in \eqref{E:wish} is an immediate consequence of the assumption that $\mu$ is absolutely continuous with respect to the $(\alpha,p)$-capacity, the other implication, namely that 
$\widetilde{u}=0$ $\mu$-a.e.\ on $\Gamma$ 
implies 
$\widetilde{u}=0$ $(\alpha,p)$-q.e.\ on $\Gamma$,  
is not satisfied in general. A counterexample where \eqref{E:wish} 
fails to hold was given in \cite[Examples 5.1.2, p. 240]{FOT2011}; we comment on it in Remark \ref{R:counter} below. 
%
%
%
%

The properties \eqref{E:kerTrmu} and \eqref{E:wish} are not only of theoretical interest. A significant motivation for the current study is the fact that if \eqref{E:kerTr} holds then \eqref{E:kerTrmu} and \eqref{E:wish} are \emph{both necessary and sufficient
for the convergence of Galerkin schemes} for the numerical solution of certain operator equations, as we will see in Section \ref{S:Galerkin} {(see also \cite[Section 4]{Ch-WCHR-PS2025})}. Such equations arise, for instance, in the study of integral equation formulations of PDEs in non-smooth domains. One particular context in which our results apply is the scattering of acoustic and electromagnetic waves by fractal obstacles, as studied for example in \cite{CCh-WCGHM25, CCh-WGHM24, Ch-WHM2017, Ch-WHMB2021,Ch-WCHR-PS2025}; for specific applications of \eqref{E:wish} in this context we mention \cite[Theorem 3.9]{CCh-WCGHM25}, \cite[Theorem 2.7]{CCh-WGHM24} and \cite[Proposition 3.6]{Ch-WCHR-PS2025}.

The most general positive results currently available for 
{\eqref{E:kerTrmu} and \eqref{E:wish}}  
appear to be those in 
\cite[Proposition 6.7]{CHM21}. They provide {\eqref{E:kerTrmu} 
(and, as a consequence, \eqref{E:wish})}  
for 
$\alpha_p(\Gamma)<\alpha<\alpha_p(\Gamma) + 1$ 
under the assumption that $\Gamma$ is a $d$-set \cite{JW84, Triebel97} for some $0<d<n$. In this case $\dim_H \Gamma=d$, $\alpha_p(\Gamma)=\frac{n-d}{p}$ and the $(\alpha_p(\Gamma),p)$-capacity of $\Gamma$ is zero. Such sets have local Hausdorff dimension $d$, uniformly on $\Gamma$. 
{A classical example is} 
where $\Gamma$ is the boundary of a Lipschitz domain, in which case $d=n-1$, so that $\alpha_p(\Gamma)=\frac{1}{p}$.
We note that \cite[Proposition 6.7]{CHM21} also provides generalisations of \eqref{E:kerTrmu} for larger $\alpha$ involving traces of partial derivatives, and for 
a class of Besov spaces. 

{
We note that the
result of \cite[Proposition 6.7]{CHM21} in the case $\alpha_p(\Gamma)<\alpha<\alpha_p(\Gamma) + 1$ was claimed previously in \cite[Proposition 19.5]{Triebel01}, under the additional assumption that the $d$-set $\Gamma$ is compact. However, as was remarked on in \cite{CHM21}, we consider the proof of \cite[Proposition 19.5]{Triebel01} to be incomplete, as it assumes \eqref{E:wish} without proper 
justification. The proof of \cite[Proposition 6.7]{CHM21} avoids this, arguing instead via extension operators and a trick used previously in \cite{Marschall87,Wallin91,FarkasJacob2001} involving approximation by continuous functions, for which \eqref{E:wish} is trivial (see Remark \ref{R:trivial} below).}

In the context of the wave scattering application mentioned above, the wish to 
consider 
scattering problems 
that are more physically realistic 
makes it desirable to have 
{\eqref{E:kerTrmu} and \eqref{E:wish}} 
for classes of sets $\Gamma$ more general than $d$-sets.  For modeling purposes it seems particularly interesting to 
{consider} 
sets $\Gamma$ which may have parts of different Hausdorff dimensions and measures $\mu$ with {(possibly discontinuously)} varying pointwise dimension.
Such more general sets $\Gamma$ were also considered in the recent studies \cite{HR-PT2021, HR-PT2023, CHR-PT24} of boundary value problems on varying domains and measure-free approaches to layer potential operators.

We are not aware of any results extending 
{\eqref{E:kerTrmu} and \eqref{E:wish}} 
beyond the $d$-set case. 
Here we provide such an extension. 
Our main results 
are stated in Sections \ref{S:results} and \ref{S:kernels}, 
see 
Theorems \ref{T:main} and \ref{T:main_logcase}, and Corollaries \ref{C:dset}, \ref{C:logHoelder}, \ref{C:trace}, \ref{C:dsettrace} and \ref{C:variant}. 
We consider supports $\Gamma=\supp \mu$ of (nonnegative) Radon measures $\mu$ on $\mathbb{R}^n$. The measures $\mu$ are assumed to be absolutely continuous with respect to $(\alpha,p)$-capacity in the aforementioned sense. The subsets $\Gamma$ of $\mathbb{R}^n$ are closed and may have zero Lebesgue measure.


{
In Theorem \ref{T:main} we prove \eqref{E:wish} 
under the assumption}  
that at each $x\in \Gamma$ the measure $\mu$ satisfies a kind of asymptotically-localized upper regularity condition, see \eqref{E:uniform} and \eqref{E:fatDbar}. This condition involves a 
{dimension function} 
$d(x)$, {which must be large enough, 
{uniformly on $\Gamma$,} 
to ensure that {any relatively open subset of} $\Gamma$ has positive $(\alpha,p)$-capacity, \eqref{E:traceconddx}, and has to satisfy {a logarithmically quantified lower-semicontinuity condition, \eqref{E:L}. 
}} 
We further assume that at $(\alpha,p)$-q.e.\ $x\in \Gamma$ the usual upper density with the 
parameter $d(x)$ is
strictly positive, \eqref{E:suffconddx}. This prevents $\mu$ from having too large \enquote{holes}, i.e., sets where the upper density degenerates.
{In Corollaries \ref{C:dset} and  \ref{C:logHoelder} we consider two examples: piecewise constant $d(x)$, and $\log$-H\"older continuous $d(x)$. 
In Theorem \ref{T:main_logcase} we provide a logarithmic refinement of Theorem \ref{T:main} in the case where $\alpha=n/p$. In Corollaries \ref{C:trace}, \ref{C:dsettrace} and \ref{C:variant} we apply our results for \eqref{E:wish} to deduce results for \eqref{E:kerTrmu} and its higher-order generalisations.} 

As well as extending far beyond the $d$-set case, our results also provide novel contributions in the $d$-set case. Specifically, we are able to extend the results from \cite[Proposition 6.7]{CHM21} to include the case $d=n$, and to include certain limiting values of $\alpha$; see Corollaries \ref{C:dset} and \ref{C:dsettrace} and Remark \ref{R:improve}. 

As mentioned above, the proof in \cite{CHM21} of \eqref{E:wish} in the $d$-set case relies on extension operators. The conceptual
novelty in our proof of \eqref{E:wish} is that we avoid them and only use fine continuity properties. The basic argument is to approximate $\widetilde{u}(x)$ by averages of $\widetilde{u}$. This idea dates back to \cite{Meyers75}, was already employed in \cite[Chapter VIII, Proposition 2]{JW84}, \cite[Theorem 1]{Wallin91} and \cite[Proposition 3.3]{Caetano2000} and mentioned in \cite[Remark 3.3.5]{Ziemer}. However, the way we use it is rather different and seems to be new. In \cite{Caetano2000, JW84, Wallin91, Ziemer} averages were taken with respect to the Lebesgue measure, but here we take averages with respect to $\mu$. At $(\alpha,p)$-q.e.\ $x$, the convergence of averages to $\widetilde{u}(x)$ is obtained from the $(\alpha,p)$-fine continuity of $\widetilde{u}$. At each such $x$, an $(\alpha,p)$-thin set has to be excluded when taking ball averages, and one has to ensure that the measure of the resulting \enquote{reduced} balls remains positive. This is done by combining the definition of $(\alpha,p)$-thinness, localized variants of well-known \enquote{isoperimetric} inequalities for $\mu$ \cite[Section 8.5]{Mazya85}, and the $(\alpha,p)$-q.e.\ strict positivity assumption on the upper density. 

The above discussions relate to trace operators defined on $H^\alpha_p(\R^n)$. There is also a considerable body of literature on trace operators mapping a Sobolev space on an open set $\Omega\subsetneqq\mathbb{R}^n$ to a function space on some closed subset of $\overline{\Omega}$, e.g.\ $\Gamma=\partial\Omega$. 
Under appropriate assumptions on $\Omega$ it is possible to derive results about the kernels of trace operators on $\Omega$ from corresponding results on $\R^n$, 
by composition with a suitable extension operator. We explore this topic in Section \ref{S:domains}. 

Our principal focus in Section \ref{S:domains} is on the fractional Sobolev spaces $H_p^\alpha(\Omega)$ defined in terms of restriction \cite[Definition 4.2.1]{Triebel78}. {A special case of Corollary \ref{C:traceop} shows that} if $\mu$ and $\Gamma$ satisfy the assumptions of Theorem \ref{T:main}, 
$\overline{\Omega}\setminus\Gamma$ is dense in $\overline{\Omega}$, and 
$\Omega$ satisfies the so-called measure density condition (see \cite{Hajlaszetal08},  \eqref{E:measuredenscond} and Remark \ref{R:measuredenscond} below), then for appropriate values of $\alpha$ the kernel of a canonical trace operator mapping $H^\alpha_p(\Omega)$ into $L^p(\Gamma,\mu)$ can be characterised as the closure of $C^\infty_c(\R^n\setminus\Gamma)$ in $H^\alpha_p(\Omega)$. 
Our proof uses the fact that the measure density condition guarantees both the existence of a bounded linear extension operator $\mathrm{E}_\Omega:H_p^\alpha(\Omega)\to H^\alpha_p(\R^n)$ with $\mathrm{E}_\Omega u|_\Omega=u$ for every $\alpha>0$ \cite[Theorem 1.1(a)]{Rychkov00}, and also, crucially, that if $u\in H^\alpha_p(\R^n)$ vanishes a.e.\ with respect to Lebesgue measure on $\Omega$ then $\tilde{u}$ vanishes $(\alpha,p)$-q.e.\ on $\overline{\Omega}$ (see {Theorem \ref{T:domain0} and} Corollary \ref{C:domain}). 
Examples of domains satisfying the measure density condition include $(\varepsilon,\delta)$-domains \cite{Jones81} and, in particular, Lipschitz domains.

For positive integer $m$ and Sobolev spaces $W_p^m(\Omega)$ defined in terms of intrinsic norms, if there is a bounded linear extension operator $\mathrm{E}_\Omega:W_p^m(\Omega)\to W^m_p(\R^n)$ with $\mathrm{E}_\Omega u|_\Omega=u$ then $W_p^m(\Omega)=H_p^m(\Omega)$ with equivalent norms, and the measure density condition is automatically satisfied \cite[Theorem 2]{Hajlaszetal08}. In this case characterizations of the kernels of trace operators on $W_p^m(\Omega)$ can be obtained from those of trace operators on $H_p^m(\Omega)$. 
It is well known that the above $W_p^m$ extension property holds for every $(\varepsilon,\delta)$-domain \cite[Theorem 1]{Jones81}, \cite[Theorem 8]{Rogers06} and, in particular, every Lipschitz domain \cite{Calderon61}, \cite[Chapter VI, Theorem 5]{Stein70}.

Regarding related literature, we note that for $\Gamma=\partial\Omega$ characterizations of the kernel of trace operators mapping from $H_p^\alpha(\Omega)$ or $W_p^m(\Omega)$ into $L^p(\partial\Omega)$ as the closure 
of $C_c^\infty(\Omega)$ in $H_p^\alpha(\Omega)$, 
respectively $W_p^m(\Omega)$,
are well known in the case of smooth domains \cite{Necas67, Triebel78}. For the spaces $W_p^m(\Omega)$ a corresponding result in the case of Lipschitz domains was provided in \cite[Theorem 1]{Marschall87} and in the case of $(\varepsilon,\delta)$-domains with a $d$-set boundary in \cite[Theorem 3]{Wallin91}.  For the spaces $H_p^\alpha(\Omega)$ with $0<\alpha\leq 1$ such a result was proved in \cite[Theorem 3.5]{FarkasJacob2001} when $\Omega$ is a bounded $(\varepsilon,\delta)$-domain with a $d$-set boundary. 
Our results in Section \ref{S:domains} provide a generalization of the above results to more general boundaries and general $\alpha>0$.

While images of trace operators are not the focus of this paper, we note briefly that characterizations of the image $[\widetilde{B}(\Gamma)]_\mu$ of $[\mathrm{Tr_{\Gamma,0}}]_\mu$ are well known for smooth submanifolds $\Gamma$, for boundaries $\Gamma$ of Lipschitz domains, for $d$-sets $\Gamma$, equipped with a $d$-measure $\mu$ \cite{JW84, Triebel97},
and even for more general closed subsets $\Gamma$ if they are the support of a measure $\mu$ satisfying certain refined doubling conditions \cite{Jonsson94}. 

The structure of the paper is as follows. In Section \ref{S:results} we recall basic notation and facts and state our main results, Theorems \ref{T:main} and \ref{T:main_logcase}. Consequences for the kernels of trace operators are discussed in Section \ref{S:kernels} and applications to domains in Section \ref{S:domains}. Implications for the convergence of Galerkin schemes for operator equations on compact sets $\Gamma$ are briefly commented on in Section \ref{S:Galerkin}. The proof of our main results is provided in Sections \ref{S:iso} and \ref{S:proof}. 

Given $1<p<\infty$, we write $p'$ for its H\"older conjugate, determined by $\frac1p+\frac{1}{p'}=1$. The symbol $B(x,r)$ denotes an open ball with center $x\in\mathbb{R}^n$ and radius $r>0$. In this article we consider real-valued functions. The main results carry over to complex-valued functions by applying them separately to their real and imaginary parts.

\stoptoc
\section*{Acknowledgements}
The first-named author thanks Takashi Kumagai for inspiring and helpful remarks.
\resumetoc

\section{{Notation and the main equivalence results}}\label{S:results}

Given {$n\geq 1$,} $1<p<\infty$ and $\alpha>0$, we write $H^\alpha_p(\mathbb{R}^n)$ for the \emph{fractional Sobolev space} (or \emph{Bessel potential space}) consisting of all $u\in L^p(\mathbb{R}^n)$ such that $((1+|\cdot|^2)^{\alpha/2} \hat u)^\vee \in L^p(\mathbb{R}^n)$; here $\varphi\mapsto \hat{\varphi}$ denotes the Fourier transform acting on tempered distributions and $\varphi\mapsto \check{\varphi}$ its inverse. Normed by 
\[\|u\|_{H^\alpha_p(\mathbb{R}^n)}:=\|((1+|\cdot|^2)^{\alpha/2} \hat u)^\vee\|_{L^p(\mathbb{R}^n)}\]
the spaces $H^\alpha_p(\mathbb{R}^n)$ are Banach spaces, for $p=2$ Hilbert. Given an open set $\Omega\subset \mathbb{R}^n$, we write $C_c^\infty(\Omega)$ for the space of all smooth functions with compact support in $\Omega$.

Given a compact set $K\subset\mathbb{R}^n$, its \emph{$(\alpha,p)$-capacity} is 
\[\cpct_{\alpha,p}(K):=\inf\big\lbrace  \|u\|_{H^\alpha_p(\mathbb{R}^n)}^p:\ u\in C_c^\infty(\mathbb{R}^n),\ u\geq 1\ \text{on $K$}\big\rbrace.\]
For open $G\subset\mathbb{R}^n$ we define $\cpct_{\alpha,p}(G):=\sup_K\cpct_{\alpha,p}(K)$, the supremum ranging over all compact $K\subset G$. For general Borel sets $E\subset \mathbb{R}^n$ we can consistently define $\cpct_{\alpha,p}(E):=\inf_G\cpct_{\alpha,p}(G)$ with the infimum taken over all open $G\supset E$. Background can be found in \cite[Section 2.2]{AH96}, a discussion of different notations and the equivalence of different definitions for $\cpct_{\alpha,p}$ in \cite[Remark 3.2]{HM2017}. A property which holds outside a set of zero $(\alpha,p)$-capacity is said to hold \emph{$(\alpha,p)$-quasi everywhere} or, for short, \emph{$(\alpha,p)$-q.e.} For $\alpha>\frac{n}{p}$ we have $H^\alpha_p(\mathbb{R}^n)\subset C(\mathbb{R}^n)$ by the Sobolev embedding theorem and, as a consequence, any nonempty set has positive $(\alpha,p)$-capacity, so that a property which holds $(\alpha,p)$-q.e.\ holds everywhere.

An extended real-valued function $v$ defined $(\alpha,p)$-q.e.\ on $\mathbb{R}^n$ is called \emph{$(\alpha,p)$-quasi continuous} if for any $\varepsilon>0$ there is an open set $G\subset\mathbb{R}^n$ with $\cpct_{\alpha,p}(G)<\varepsilon$ outside of which $v$ is continuous. For every $u\in H^\alpha_p(\mathbb{R}^n)$ we can find an $(\alpha,p)$-quasi continuous representative $\widetilde{u}$, that is, an $(\alpha,p)$-quasi continuous Borel function on $\mathbb{R}^n$ such that $\widetilde{u}=u$ $\mathcal{L}^n$-a.e. Here and below $\mathcal{L}^n$ denotes $n$-dimensional Lebesgue measure. Two $(\alpha,p)$-quasi continuous representatives $\widetilde{u}$ of $u \in H^\alpha_p(\mathbb{R}^n)$ can differ only on a set of zero $(\alpha,p)$-capacity. {See \cite[Propositions 2.3.7, 6.1.2, 6.1.3 and Theorem 6.1.4]{AH96} for these properties.} For $\alpha>\frac{n}{p}$ an $(\alpha,p)$-quasi continuous function is continuous, so that there is exactly one $(\alpha,p)$-quasi continuous representative $\widetilde{u}$ of $u\in H_p^\alpha(\mathbb{R}^n)$, and $\tilde u$ is continuous.

We call a non-decreasing and right-continuous function $h:[0,1)\to [0,\infty)$ a \emph{Hausdorff function} if $h(0)=0$, $h$ is strictly positive on $(0,1)$, and for some $c>1$ we have $h(2r)\leq c\,h(r)$, $0<r<\frac12$. We write $\mathcal{H}^h$ for the \emph{Hausdorff measure associated with $h$} \cite{AH96, Mazya85,Zaehle}. If $0\leq d\leq n$ and $h(r)=r^d$, $0\leq r<1$, then we write $\mathcal{H}^d$; this is \emph{$d$-dimensional Hausdorff measure}. 

{ 
Let $\mu$ be a Radon measure \cite[Definition 1.5]{Mattila} on $\mathbb{R}^n$ and let $x\in\mathbb{R}^n$. Given a Hausdorff function $h$, the \emph{upper $h$-density of $\mu$ at $x$} is defined as
\[\bar{D}^{h}\mu(x):=\limsup_{r\to 0}\frac{\mu(B(x,r))}{h(r)}.\]
{Set $\Gamma:=\supp\mu$ and let $d:\Gamma\to [0,\infty)$  be a Borel function.} If $0< d(x)\leq n$ and $h(r)=r^{d(x)}$, $0\leq r<1$, then $\bar{D}^{h}\mu(x)$ is  called the \emph{upper $d(x)$-density of $\mu$ at} $x$ and is denoted by $\bar{D}^{d(x)}\mu(x)$. See \cite[Definition 6.8]{Mattila} and \cite[Definition 3.1]{Zaehle}. We also consider a kind of locally uniform variant of the upper $d(x)$-density. Given $x\in \Gamma$ and $0<r\leq 1$, we set
\[\underline{d}(x,r):=\inf_{y\in B(x,r)\cap \Gamma}d(y)\]
and write
\begin{equation}\label{E:uniform}
\bar{\mathbb{D}}^{d(x)}\mu(x):=\limsup_{r\to 0} \sup_{y\in B(x,r)\cap \Gamma,\ 0<\varrho\leq r}\frac{\mu(B(y,\varrho))}{\varrho^{\underline{d}(x,r)}}.
\end{equation}
We also introduce the quantity 
\[L_d(x){:=\limsup_{r\to 0}\big(d(x)-\underline{d}(x,r)\big)(-\log r)}=\limsup_{r\to 0}\sup_{y\in B(x,r)\cap \Gamma}\big(d(x)-d(y)\big)(-\log r).
\]


%

{All the results in this section hold for all dimensions $n\geq 1$.} Our main result is the following. 
 
\begin{theorem}\label{T:main}
Let $\mu$ be a non-zero Radon measure on $\mathbb{R}^n$, $\Gamma=\supp\mu$ and $1<p<\infty$. Suppose that $0<\alpha\leq \frac{n}{p}$ and that $d:\Gamma\to [0,\infty)$ is a Borel function such that 
\begin{equation}\label{E:L}
L_d(x)<\infty,\qquad x\in \Gamma,
\end{equation}
\begin{equation}\label{E:fatDbar}
\bar{\mathbb{D}}^{d(x)}\mu(x)<\infty, \qquad x\in \Gamma,
\end{equation}
\begin{equation}\label{E:traceconddx}
n-\alpha p<\inf_{x\in \Gamma}d(x),
\end{equation}
 and 
\begin{equation}\label{E:suffconddx}
\cpct_{\alpha,p}(S_{d,\mu}(\Gamma))=0,
\end{equation}
where 
\begin{equation}\label{E:Sd}
S_{d,\mu}(\Gamma) := \{x\in\Gamma:\ \bar{D}^{d(x)}\mu(x)=0\}.
\end{equation}  
Then the equivalence \eqref{E:wish} holds for any $u\in H^\alpha_p(\mathbb{R}^n)$. 
\end{theorem}

\begin{remark}\label{R:trivial} For convenience we require that $\alpha\leq \frac{n}{p}$ in our statement of Theorem \ref{T:main}. But, as noted above, if $\alpha>\frac{n}{p}$ then for any  $u\in H^\alpha_p(\mathbb{R}^n)$ the (unique) $(\alpha,p)$-quasi continuous representative $\widetilde{u}$ is continuous, so that  {\eqref{E:wish} is straightforward from the fact that $\Gamma=\supp\mu$.}
\end{remark}

{\begin{remark}\label{R:basicimplication}\mbox{}
\begin{enumerate}
\item[(i)] If $L_d(x)<\infty$ at a point $x\in \Gamma$, then 
\[\bar{D}^{d(x)}\mu(x)=\limsup_{r\to 0} \frac{\mu(B(x,r))}{r^{d(x)}}\leq \limsup_{r\to 0}r^{\underline{d}(x,r)-d(x)}\limsup_{r\to 0} \frac{\mu(B(x,r))}{r^{\underline{d}(x,r)}}\leq e^{L_d(x)}\bar{\mathbb{D}}^{d(x)}\mu(x).\]
In particular, if \eqref{E:L} and \eqref{E:fatDbar} hold, 
then also
\begin{equation}\label{E:Dbar}
\bar{D}^{d(x)}\mu(x)<\infty, \qquad x\in \Gamma.
\end{equation}
\item[(ii)] {If \eqref{E:L} holds}, then $d$ is lower semicontinuous {and $d(x)=\lim_{r\to 0} \underline{d}(x,r)$ at each $x\in \Gamma$.} 
\end{enumerate}
\end{remark}
}

The first important special case is that of a constant function $d(\cdot)\equiv d$. If $0<d\leq n$ is constant and 
\begin{equation} \label{E:dset}
c^{-1}r^d\leq \mu(B(x,r))\leq cr^d, \quad 
x\in \Gamma=\supp \mu, \quad 0<r\leq 1,
\end{equation} 
where $c>1$ is a fixed constant, then $\mu$ is called a \emph{$d$-measure} or an \emph{Ahlfors $d$-regular measure}. {Recall that if for a given set $\Gamma\subset \mathbb{R}^n$ 
there is a Radon measure $\mu$ such that \eqref{E:dset} holds, then  $\Gamma$ is called a \emph{$d$-set};} see \cite[Section II.1]{JW84} or \cite[Definition 3.1]{Triebel97}. In this case $\dim_H\Gamma=d$.

\begin{remark}\label{R:main}\mbox{}
\begin{enumerate}
\item[(i)] If $\Gamma$ is a $d$-set, then \eqref{E:dset} holds (for some possibly different $c>1$) with $\mu$ replaced by $\cH^d|_\Gamma$, see \cite[Theorem 3.4]{Triebel97}. For our calculations below we note that, if the second of the inequalities in \eqref{E:dset} holds for $x\in \Gamma$ and $0<r\leq 1$, then, by a simple triangle inequality argument, also  $\mu(B(x,r))\leq c\:2^dr^d$ for all $x\in \R^n$ and $0<r\leq \frac{1}{2}$.  
\item[(ii)] If $\mu$ is a $d$-measure and \eqref{E:dset} holds, then $\mu$ satisfies the doubling condition 
\begin{equation}\label{E:usualdoubling}
\mu(B(x,2r))\leq c_D\mu(B(x,r)),\qquad x\in\Gamma,\quad 0<r\leq \frac12, 
\end{equation}
with a constant $c_D>0$. Here we can use $c_D=c^22^d$.
\end{enumerate}
\end{remark}

\begin{examples}{ 
It is well known that all attractors of iterated function systems of contracting similarities satisfying the open set condition are $d$-sets \cite{Moran46, Hutchinson81, Falconer97}. The same is true for attractors of conformal iterated function systems, see \cite{MauldinUrbanski96, Zaehle2001}.}
\end{examples}

\begin{corollary}\label{C:dset0}
{Let $1<p<\infty$, $0<d\leq n$, let $\mu$ be a $d$-measure and $\Gamma=\supp \mu$. If $\alpha>\frac{n-d}{p}$, then \eqref{E:wish} holds for any $u\in H^\alpha_p(\mathbb{R}^n)$.}
\end{corollary}
\begin{proof}
Since $d$ is constant, $\underline{d}(x,r)=d$ and $L_d(x)=0$ for all $x\in \Gamma$ and $0<r\leq 1$. Condition \eqref{E:dset} gives $c^{-1}\leq \bar{D}^d\mu(x)\leq \bar{\mathbb{D}}^d\mu(x)\leq c$ for all $x\in\Gamma$. {Consequently $\mu$ and $d$ satisfy the conditions \eqref{E:L}, \eqref{E:fatDbar}, \eqref{E:traceconddx}, and \eqref{E:suffconddx} of Theorem \ref{T:main} with $S_{d,\mu}(\Gamma)=\emptyset$. This gives the result for $\alpha\leq \frac{n}{p}$, Remark \ref{R:trivial} gives it for $\alpha>\frac{n}{p}$.} 
\end{proof}

\begin{remark}\label{R:main2}
If $0<d<n$, $\frac{n-d}{p}<\alpha<\frac{n-d}{p}+1$, $\mu$ is a $d$-measure and $\Gamma=\supp \mu$, then the validity of \eqref{E:wish} for any $u\in H^\alpha_p(\mathbb{R}^n)$ also follows from an application of \cite[Proposition 6.7]{CHM21} in the special case of fractional Sobolev spaces, together with \cite[Corollary 10.1.2]{AH96}. {Corollary \ref{C:dset0} extends this result in two ways: it includes the case $d=n$, and, in the case that $\frac{n-d}{p}+1\leq \frac{n}{p}$, i.e.\ $d\geq p$, it extends the range of $\alpha$ for which \eqref{E:wish} holds.} The results provided in \cite[Section 6]{CHM21} are for both fractional Sobolev spaces and Besov spaces. If $p=2$ and $\alpha=\frac{n-d}{2}$, then \eqref{E:wish} follows from \cite[Corollary 18.12]{Triebel97} together with \cite[Remark 3.5 (i) and Proposition 4.3]{Caetano2002}. 
\end{remark}

Theorem \ref{T:main} also applies to measures $\mu$ whose support $\Gamma$ has parts of different Hausdorff dimensions.

\begin{examples}\label{Ex:hybrid} Let $\Gamma_1\subset \mathbb{R}^2$ be the classical Koch curve \cite{Falconer97}, located so that it has starting point $(-1,0)$ and endpoint $(0,0)$, and let $\mu_1=\mathcal{H}^{d}|_{\Gamma_1}$ with $d=\frac{\log 4}{\log 3}$. Let $\Gamma_2=[0,1]\times \{0\}$, endowed with $\mu_2=\mathcal{H}^1|_{\Gamma_2}$, and consider $\mu:=\mu_1+\mu_2$. Obviously $\Gamma:=\supp{\mu}=\Gamma_1\cup\Gamma_2$. If
\[d(x)=\begin{cases} d & \text{for $x\in \Gamma_1\setminus \{(0,0)\}$,} \\ 1 & \text{for $x\in \Gamma_2$,}\end{cases}\]
then for any $x\in \Gamma$ we have $L_d(x)=0$ and $0<\bar{D}^{d(x)}\mu(x)\leq \bar{\mathbb{D}}^{d(x)}\mu(x)<\infty$. 
\end{examples}

The following consequence of Theorem \ref{T:main} generalizes both Corollary \ref{C:dset0} and {Example} \ref{Ex:hybrid}.

\begin{corollary}\label{C:dset} {
Let $1<p<\infty$. Suppose that $J$ is a positive integer, and, for $j=1,\ldots, J$, that $0<d_j\leq n$, $\Gamma_j\subset \R^n$ is a $d_j$-set, $n-\alpha p <d_j$ and $a_j>0$. Let
$$
\mu := \sum_{j=1}^J a_j\cH^{d_j}|_{\Gamma_j}, \qquad \Gamma := \bigcup_{j=1}^J \Gamma_j = \supp \mu,
$$ 
and, for $x\in \Gamma$, let $I(x):= \{j\in \{1,\ldots,J\}:x\in \Gamma_j\}$ and $d(x):= \min\{d_j:j\in I(x)\}$. Then \eqref{E:wish} holds for any $u\in H^\alpha_p(\mathbb{R}^n)$.}
\end{corollary}

\begin{proof} Clearly $d$ is a Borel function, with $\inf_{x\in \Gamma}d(x)=\min\{d_1,\ldots,d_J\}>n-\alpha p$.  Since all $\Gamma_j$ are closed, each $x\in \Gamma$ has a minimal distance $r(x)>0$ to all $\Gamma_j$ with $j\notin I(x)$, so that $\underline{d}(x,r)=d(x)$ for all $r<r(x)$. Therefore  $L_d(x)=0$, $x\in \Gamma$. As noted in Remark \ref{R:main} (i), since each $\Gamma_j$ is a $d_j$-set,  equation \eqref{E:dset} holds for some $c>1$ with $d$, $\Gamma$, and $\mu$ replaced by $d_j$, $\Gamma_j$, and $\cH^{d_j}|_{\Gamma_j}$, respectively, for $j=1,\ldots, J$, indeed, if $0<r\leq 1/2$, the right-hand inequality in \eqref{E:dset} holds for all $x\in \R^n$ and all $j=1,...,N$ if $c$ is replaced by $c2^d$. Thus,
 for $x\in \Gamma$, where $\mu_j:= a_j\cH^{d_j}|_{\Gamma_j}$,
$$
\bar{\mathbb{D}}^{d(x)}\mu(x) \leq  \sum_{j=1}^J \bar{\mathbb{D}}^{d_j}\mu_j(x)  \leq  c2^d\sum_{j=1}^J a_j.
$$
Similarly, where $\underline{a}:=\min\{a_1,\ldots,a_J\}$, 
$\bar{D}^{d(x)}\mu(x) \geq c^{-1}\underline{a}$, for $x\in \Gamma$, since $I(x)\neq \emptyset$ for every $x\in \Gamma$. {Consequently $\mu$ and $d$ satisfy the conditions \eqref{E:L}, \eqref{E:fatDbar}, \eqref{E:traceconddx} and \eqref{E:suffconddx} of Theorem \ref{T:main} with $S_{d,\mu}(\Gamma)=\emptyset$. Together with Remark \ref{R:trivial}, the result follows.}
\end{proof}

\begin{remark} 
Unless all $d_j$ are equal, the measures $\mu$ in Corollary \ref{C:dset} fail to satisfy any doubling condition \eqref{E:usualdoubling}. In particular, the measure $\mu$ in 
{Example} \ref{Ex:hybrid} is not doubling. 
\end{remark}

Theorem \ref{T:main} also applies to {certain non-constant} continuous functions $d$. In \cite[Lemma 2.1]{Harjulehto2006} it was proved that if a Radon measure $\mu$ is such that
\begin{equation} \label{E:dxset}
c^{-1}r^{d(x)}\leq \mu(B(x,r))\leq cr^{d(x)}, \quad 
x\in \Gamma=\supp \mu, \quad 0<r\leq 1,
\end{equation} 
where $c>1$ is a fixed constant and $d:\Gamma\to (0,\infty)$ is a bounded Borel function, then 
\begin{equation}\label{E:logHoelderglobal}
|d(x)-d(y)|(-\log|x-y|)<c^22^{\bar{d}}
\end{equation}
for all $x,y\in \Gamma$ with $|x-y|<\frac12$; here $\bar{d}:=\sup_{x\in \Gamma}d(x)$. 

\begin{remark}
{Obviously} \eqref{E:dxset} generalizes \eqref{E:dset}. Any measure $\mu$ as in \eqref{E:dxset} satisfies the doubling condition \eqref{E:usualdoubling} with $c_D=c^2 2^{\bar{d}}$, as already pointed out in \cite{Harjulehto2006}. 
\end{remark}

{ 
\begin{corollary}\label{C:logHoelder} {
Let $1<p<\infty$, let $\mu$ be a non-zero Radon measure on $\mathbb{R}^n$, $\Gamma=\supp\mu$, and $\alpha>0$. If $\alpha\leq \frac{n}{p}$, suppose that $d:\Gamma\to [0,\infty)$ be a bounded Borel function such that \eqref{E:traceconddx} and \eqref{E:dxset} hold. Then \eqref{E:wish} holds for any $u\in H^\alpha_p(\mathbb{R}^n)$. }
\end{corollary} 

\begin{proof} {By Remark \ref{R:trivial} it suffices to discuss the case $\alpha\leq \frac{n}{p}$.} For any $x\in \Gamma$ and $0<r\leq 1$, the right inequality in \eqref{E:dxset} gives
$\mu(B(y,\varrho))\leq c \varrho^{d(y)}\leq c\varrho^{\underline{d}(x,r)}$ for all $y\in B(x,r)\cap \Gamma$ and $0<\varrho\leq r$, so that 
$\bar{\mathbb{D}}^{d(x)}\mu(x)<\infty$. Obviously \eqref{E:logHoelderglobal} implies that
\[L_d(x)\leq \limsup_{r\to 0}\sup_{y\in B(x,r)\cap \Gamma}|d(x)-d(y)|(-\log r)<\infty,\qquad x\in\Gamma.\]
By the left inequality in \eqref{E:dxset} we have $\bar{D}^{d(x)}\mu(x)>c^{-1}$ at all $x\in \Gamma$. {It follows that $\mu$ and $d$ also satisfy the conditions \eqref{E:L}, \eqref{E:fatDbar} and \eqref{E:suffconddx} of Theorem \ref{T:main} with $S_{d,\mu}(\Gamma)=\emptyset$.}
\end{proof}
}

\begin{examples}\label{L:logHoelder} 
Given a Lipschitz function $s:[0,1]\to (\frac14,\frac12)$, {it was shown in \cite[Theorem 3.4]{Harjulehto2006}} that one can construct a Koch curve $\Gamma\subset \mathbb{R}^2$, a continuous bijection $\Phi:[0,1]\to\Gamma$ and a finite Radon measure $\mu$ such that \eqref{E:dxset} holds with {the bounded Borel function $d:\Gamma\to [0,\infty)$ defined by }
\[d(x):=\frac{\log 4}{-\log s(\Phi^{-1}(x))},\qquad x\in \Gamma.\]
{Clearly, \eqref{E:traceconddx} also holds as long as 
\[\frac{\log 4}{-\log\min_{0\leq t\leq 1}s(t)}>2-\alpha p.\]
}
\end{examples}

\begin{remark}\label{R:upregd} Let $\mu$ be a non-zero Radon measure on $\mathbb{R}^n$ and $\Gamma=\supp\mu$ and let $d:\Gamma\to [0,\infty)$ be a Borel function.
\begin{enumerate}
\item[(i)] {If \eqref{E:fatDbar} holds, then} for each $x\in \Gamma$ there is some $r(x)>0$ such that for all {
$0<\varrho\leq r\leq r(x)$ and all {$y\in B(x,r)\cap \Gamma$}} we have 
\begin{equation}\label{E:upregdx}
\mu(B(y,\varrho))\leq c(x){ \varrho^{\underline{d}(x,r)}}, 
\end{equation}
where {$c(x):=2\:\bar{\mathbb{D}}^{d(x)}\mu(x)$ if $\bar{\mathbb{D}}^{d(x)}\mu(x)>0$ and $c(x):=1$ otherwise.} Together with \cite[Theorem 6.9]{Mattila} this implies that $\underline{d}(x,r)\leq \dim_H(B(x,r)\cap \Gamma)\leq \dim_H\Gamma$ for all $x\in \Gamma$ and $r\leq r(x)$. {If in addition \eqref{E:L} holds,} {then by Remark \ref{R:basicimplication} (ii) we find that} 
\begin{equation}\label{E:leqdimH}
\sup_{x\in\Gamma}d(x)\leq \dim_H\Gamma\leq n.
\end{equation}
If the set $\Gamma$ is compact, then it admits a finite cover $\{B(x_i,r(x_i))\}_{i=1}^N$ by open balls with $x_i\in \Gamma$, $i=1,...,N$. {If \eqref{E:fatDbar} holds, then} $\mu$ is \emph{upper $d_{\min}:=\min_i \underline{d}(x_i,r(x_i))$-regular}, more precisely, we have 
$\mu(B(y,\varrho))\leq c_{\max}\,\varrho^{d_{\min}}$
for all $y\in\Gamma$ and $0<\varrho\leq r_{\min}$, where $c_{\max}:=\max_i c(x_i)$ and $r_{\min}:=\min_i r(x_i)$. 
\item[(ii)] Recall \eqref{E:Sd}. Since $\bar{D}^{d(x)}\mu(x)>0$ for all $x\in \Gamma\setminus S_{d,\mu}(\Gamma)$, we have $\bar{D}^s\mu(x)=\infty$ for all $x\in  \Gamma\setminus S_{d,\mu}(\Gamma)$ and $s>\sup_{x\in \Gamma\setminus S_{d,\mu}(\Gamma)}d(x)$. By \cite[Theorem 6.9]{Mattila} it follows that for {any such $s$ and} any open ball $B$ we have $\mathcal{H}^s(B\cap \Gamma\setminus S_{d,\mu}(\Gamma))=0$, which implies that $\mathcal{H}^s(\Gamma \setminus S_{d,\mu}(\Gamma))=0$. Therefore 
\begin{equation}\label{E:geqdimH}
\dim_H(\Gamma\setminus S_{d,\mu}(\Gamma))\leq \sup_{x\in \Gamma\setminus S_{d,\mu}(\Gamma)}d(x).
\end{equation}
{If \eqref{E:L} and \eqref{E:fatDbar} hold and $S_{d,\mu}(\Gamma)=\emptyset$, then \eqref{E:leqdimH} and \eqref{E:geqdimH} combine to
\begin{equation}\label{E:dimH}
\dim_H\Gamma=\sup_{x\in\Gamma}d(x).
\end{equation}}
\item[(iii)] If $0<\alpha\leq \frac{n}{p}$ and \eqref{E:Dbar} and \eqref{E:traceconddx} are assumed, then condition \eqref{E:suffconddx} implies that
\begin{equation}\label{E:notquite}
\mu(S_{d,\mu}(\Gamma))=0,
\end{equation}
cf. Lemma \ref{L:smooth} below. The converse implication is false, {even if $d$ is constant}. To see this, suppose that $d(\cdot)\equiv n=2$, $p=2$ and $\alpha=1$. Let $\Gamma=[0,1]^2$ and let $S\subset [0,1]^2$ be such that $\mathcal{L}^2(S)=0$ but  $\cpct_{1,2}(S)>0$.
Let $\varphi:[0,1]^2\to [0,\infty)$ be a continuous function with zero set $S$. Then the measure $\mu=\varphi\cdot\mathcal{L}^2|_{[0,1]^2}$ has support $\Gamma=[0,1]^2$, and we have $c^{-1}\varphi(x)\leq \bar{D}^2\mu(x)\leq c\:\varphi(x)$, $x\in [0,1]^2$, for some constant $c>1$. {Obviously {\eqref{E:L} and }
\eqref{E:traceconddx} hold, and since $\varphi$ is bounded on $\Gamma$ also \eqref{E:fatDbar} and \eqref{E:Dbar}.} In this situation we have $\mu(\{\varphi=0\})=0$, which is \eqref{E:notquite}, but \eqref{E:suffconddx} does not hold. As one example we may choose $S=[0,1]\times \{0\}$ and $\varphi(x)=x_2$, $x=(x_1,x_2)\in [0,1]^2$. {A variant of this argument gives simple examples where $S_{d,\mu}(\Gamma)$ is nonempty but of zero capacity $\cpct_{1,2}$; we may for instance take $\varphi(x)=|x-(0,\frac12)|$.}
\item[(iv)] One refers to the quantity
\[\underline{\dim}\,\mu(x):=\liminf_{r\to 0}\frac{\log \mu(B(x,r))}{\log r}\]
as the \emph{lower pointwise dimension of $\mu$} at $x$ \cite{Falconer97, Zaehle}. It is well known that $\bar{D}^{d(x)}\mu(x)<\infty$ implies $d(x)\leq \underline{\dim}\,\mu(x)$, while $0<\bar{D}^{d(x)}\mu(x)$ implies  $\underline{\dim}\,\mu(x)\leq d(x)$ \cite[Section 6.1]{Zaehle}. Therefore, if \eqref{E:Dbar} holds, then 
\begin{equation}\label{E:lowerpointwise}
d(x)= \underline{\dim}\,\mu(x) ,\qquad x\in \Gamma\setminus S_{d,\mu}(\Gamma).
\end{equation}
{If \eqref{E:dxset} holds, then at each $x\in \Gamma$ the \emph{pointwise dimension} $\dim\mu(x):=\lim_{r\to 0} \frac{\log \mu(B(x,r))}{\log r}$ \emph{of $\mu$} exists and equals $d(x)$.}
\item[(v)] {In the case $\alpha=\frac{n}{p}$  
Theorem \ref{T:main} remains true with \eqref{E:L} replaced by \eqref{E:Dbar} and with \eqref{E:suffconddx} replaced by the weaker condition $\cpct_{\alpha,p}(\{x\in\Gamma:\underline{\dim}\,\mu(x)> d(x) \})=0$. See the proof of Lemma \ref{L:lessthanone}.}
\end{enumerate}
\end{remark}

We state a logarithmic {variant} of Theorem \ref{T:main} for the case $\alpha=\frac{n}{p}$. Note that for each $\kappa>0$ the function $h_\kappa(r)=(-\log r)^{-\kappa}$ is a Hausdorff function. Given a Borel function $\kappa:\Gamma\to [0,\infty)$, set 
\[\underline{\kappa}(x,r):=\inf_{y\in {B(x,r)\cap}\Gamma}{\kappa(y)}\quad\text{and}\quad \bar{\mathbb{D}}^{h_{\kappa(x)}}\mu(x):=\limsup_{r\to 0}\sup_{y\in B(x,r){ \cap \Gamma},\ 0<\varrho\leq r}{ \frac{\mu(B(y,\varrho))}{(-\log \varrho)^{-\underline{\kappa}(x,r)}}},\quad x\in\Gamma,\]
and let 
\[\ell_\kappa(x):=\limsup_{r\to 0}\big(\kappa(x)-\underline{\kappa}(x,r)\big)\log(-\log r).
\]

\begin{theorem}\label{T:main_logcase}
Let $\mu$ be a non-zero Radon measure on $\mathbb{R}^n$, $\Gamma=\supp\mu$ and $1<p<\infty$. Suppose that $\alpha=\frac{n}{p}$ and that $\kappa:\Gamma\to  [0,\infty)$ is a Borel function such that
\begin{equation}\label{E:Lkappa}
\ell_\kappa(x)<\infty,\qquad x\in \Gamma,
\end{equation}
\begin{equation}\label{E:fatDbarkappa}
{ \bar{\mathbb{D}}^{h_{\kappa(x)}}\mu(x)}<\infty, \qquad x\in \Gamma,
\end{equation}
\begin{equation}\label{E:tracecondhx}
p-1<\inf_{x\in \Gamma}\kappa(x)
\end{equation}
and 
\begin{equation}\label{E:suffcondh}
\cpct_{\alpha,p}(\{x\in\Gamma:\ \bar{D}^{h_{\kappa(x)}}\mu(x)=0\})=0.
\end{equation}
Then \eqref{E:wish} holds for any $u\in H^\alpha_p(\mathbb{R}^n)$.
\end{theorem}

{ 
\begin{remark}\label{R:basicimplicationkappa}\mbox{}
\begin{enumerate}
\item[(i)] If $\ell_\kappa(x)<\infty$, then 
\[{\bar{D}^{h_{\kappa(x)}}\mu(x)\leq \limsup_{r\to 0}(-\log r)^{\kappa(x)-\underline{\kappa}(x,r)}\limsup_{r\to 0} \frac{\mu(B(x,r))}{(-\log r)^{-\underline{\kappa}(x,r)}}\leq e^{\ell_\kappa(x)}\bar{\mathbb{D}}^{h_{\kappa(x)}}\mu(x).}\]
In particular, if \eqref{E:Lkappa} and \eqref{E:fatDbarkappa} hold, then also
\begin{equation}\label{E:Dbarkappa}
{ \bar{D}^{h_{\kappa(x)}}\mu(x)}<\infty, \qquad x\in \Gamma.
\end{equation}
\item[(ii)] Clearly also $\kappa$ is lower semicontinuous, {provided that \eqref{E:Lkappa} holds.}
\end{enumerate}
\end{remark}
}

\begin{remark}\label{R:upreghx} {
Under condition \eqref{E:fatDbarkappa} there is some $r(x)>0$ such that for all $0<\varrho\leq r\leq r(x)$ and all $y\in B(x,r)$ we have 
\begin{equation}\label{E:upreghx}
\mu(B(y,\varrho))\leq c(x)(-\log \varrho)^{-\underline{\kappa}(x,r)}, 
\end{equation}
where $c(x)=\max\{2\bar{\mathbb{D}}^{h_{\kappa(x)}}\mu(x),1\}$.} 
\end{remark}

{Well-known results ensure} that under the assumptions of Theorems \ref{T:main} {or} \ref{T:main_logcase} any Borel set of zero capacity is a $\mu$-null set. 

\begin{lemma}\label{L:smooth}
Let $\mu$ be a  non-zero Radon measure on $\mathbb{R}^n$, $\Gamma=\supp\mu$ and $1<p<\infty$. If
\begin{enumerate}
\item[(i)] $0<\alpha\leq \frac{n}{p}$ and $d:\Gamma\to [0,\infty)$ is a Borel function such that {\eqref{E:Dbar} and \eqref{E:traceconddx} hold} or
\item[(ii)] $\alpha=\frac{n}{p}$ and $\kappa:\Gamma\to  [0,\infty)$ is a Borel function such that {\eqref{E:Dbarkappa} and \eqref{E:tracecondhx} hold,} 
\end{enumerate}
then, for all Borel sets $E\in \mathbb{R}^n$ with $\cpct_{\alpha,p}(E)=0$, we have $\mu(E)=0$.
\end{lemma}

\begin{proof} {For convenience we briefly state the standard arguments. Recall that $p'$ denotes the conjugate of $p$.}
By \eqref{E:traceconddx} there is a number $n-\alpha p<d_0<\inf_{x\in \Gamma}d(x)$, and by {\eqref{E:Dbar}} we have $\bar{D}^{d_0}\mu(x)=0$, $x\in \Gamma$. Similarly, by \eqref{E:tracecondhx} we can find a number $p-1<\kappa_0<\inf_{x\in\Gamma}\kappa(x)$, and {\eqref{E:Dbarkappa}} gives  $\bar{D}^{h_{\kappa_0}}\mu(x)=0$, $x\in \Gamma$. If $\cpct_{\alpha,p}(E)=0$, then by \cite[Theorem 5.1.13 and the Remark following it]{AH96} we have $\mathcal{H}^{d_0}(E)=0$, respectively $\mathcal{H}^{h_{\kappa_0}}(E)=0$, noting that $\int_0^1 h_{\kappa_0}(r)^{p'-1}\frac{dr}{r}<\infty$. {In either case this implies $\mu(E)=0$ by \cite[Theorem 6.9]{Mattila}, respectively \cite[Theorem 3.7]{Zaehle}.}
\end{proof}


It is well known that any set of $(\alpha,p)$-capacity zero is contained in a Borel set of $(\alpha,p)$-capacity zero, see for instance \cite[Proposition 2.3.7]{AH96}. We therefore see that if the hypotheses of Lemma \ref{L:smooth} are satisfied, then for any $u\in H^\alpha_p(\mathbb{R}^n)$ the condition $\widetilde{u}=0$ q.e.\ on $\Gamma$ implies that  $\widetilde{u}=0$ on $\Gamma\setminus N$, where $N$ is a Borel set with $\mu(N)=0$. This gives one implication in \eqref{E:wish}. We will prove the converse implication under the additional assumption \eqref{E:suffconddx}, respectively \eqref{E:suffcondh}. The proof uses some further ideas and will be given in Sections \ref{S:iso} and \ref{S:proof}. Without any additional assumption the validity of this converse implication cannot be expected, as the following counterexample from \cite{FOT2011} shows.

\begin{remark}\label{R:counter}
Recall that a set $E\subset \mathbb{R}^n$ is called \emph{$(\alpha,p)$-quasi open} if for any $\varepsilon>0$ there is an open set $G\supset E$ such that $\cpct_{\alpha,p}(G\setminus E)<\varepsilon$. A set $E\subset \mathbb{R}^n$ is called \emph{$(\alpha,p)$-quasi closed} if its complement is $(\alpha,p)$-quasi open. If the symmetric difference $E_1\Delta E_2$ of two sets $E_1\subset \mathbb{R}^n$ and $E_2\subset \mathbb{R}^n$ has zero $(\alpha,p)$-capacity, then $E_1$ and $E_2$ are said to be \emph{$(\alpha,p)$-equivalent}, cf. \cite[p. 176]{AH96}.

In the special case $p=2$ and $0<\alpha\leq 1$ the scalar product associated with $\|\cdot\|_{H^\alpha_2(\mathbb{R}^n)}$ is a Dirichlet form \cite{FOT2011}. Suppose that we are in this Dirichlet form case and that $\mu$ is a  non-zero Radon measure such that $\mu(E)=0$ for any set of zero $(\alpha,2)$-capacity. A set $\widetilde{\Gamma}\subset \mathbb{R}^n$ is said to be a \emph{quasi support} for $\mu$ if (a) $\widetilde{\Gamma}$ is $(\alpha,2)$-quasi closed and $\mu(\mathbb{R}^n\setminus \widetilde{\Gamma})=0$, and (b) for any other set $\check{\Gamma}\subset \mathbb{R}^n$ satisfying (a) we have $\cpct_{\alpha,2}(\widetilde{\Gamma}\setminus \check{\Gamma})=0$. See \cite[p. 841]{FukushimaLeJan} or \cite[p. 190]{FOT2011}. In \cite[Theorem 3.3.(i)]{FukushimaLeJan} it was proved that any $\mu$ as specified admits a quasi support, see also \cite[Theorem 4.6.3]{FOT2011}. This quasi support is uniquely determined up to $(\alpha,2)$-equivalence. In \cite[Theorem 3.3.(ii)]{FukushimaLeJan} it was shown that for an $(\alpha,2)$-quasi closed set $\widetilde{\Gamma}\subset \mathbb{R}^n$ the property of being a quasi support for $\mu$ is equivalent to the property that for any $u\in H^\alpha_2(\mathbb{R}^n)$ we have $\widetilde{u}=0$ $\mu$-a.e.\ on $\widetilde{\Gamma}$ if and only if $\widetilde{u}=0$ $(\alpha,2)$-q.e.\ on $\widetilde{\Gamma}$. See also \cite[Theorem 4.6.2]{FOT2011}. Since any closed set is $(\alpha,2)$-quasi closed, condition \eqref{E:wish} is equivalent to saying that the support $\Gamma$ of $\mu$ is a quasi support for $\mu$.

An example of a  non-zero Radon measure $\mu$ on $\mathbb{R}^n$, $n\geq 3$, charging no set of zero $(1,2)$-capacity and having support $\Gamma=\mathbb{R}^n$ and a quasi support $\widetilde{\Gamma}\subset \mathbb{R}^n$ such that $\cpct_{1,2}(\Gamma\setminus \widetilde{\Gamma})>0$ was given in \cite[Examples 5.1.2, p. 240]{FOT2011}. Together with the preceding this is a counterexample for \eqref{E:wish} in the case $p=2$ and $\alpha=1$. The measure $\mu$ in \cite[Examples 5.1.2]{FOT2011} has  an $\mathcal{L}^n$-a.e.\ bounded density with respect to $\mathcal{L}^n$ and therefore satisfies {\eqref{E:L}, \eqref{E:fatDbar} and \eqref{E:traceconddx}} with $d\equiv n$. Since $\mu(\Gamma\setminus \widetilde{\Gamma})=0$ we have $\bar{D}^n\mu(x)=0$ at $\mathcal{L}^n$-a.e.\ $x\in \Gamma\setminus \widetilde{\Gamma}$ \cite[Theorem 2.12]{Mattila}. 
The proof that $\Gamma\setminus \widetilde{\Gamma}$ has positive $(1,2)$-capacity proceeds by showing that $\mathcal{L}^n(\Gamma\setminus \widetilde{\Gamma})>0$. Consequently $\cpct_{1,2}(\{x\in \Gamma: \bar{D}^n\mu(x)=0\})>0$, which violates \eqref{E:suffconddx}.
\end{remark}

\section{Kernels of trace operators}\label{S:kernels}

{The results in this section hold for all $n\geq 1$.} Let $1<p<\infty$ and $\alpha>0$. We write
\[m_s:=\lceil\alpha\rceil - 1,\]
where $\lceil\cdot \rceil$ denotes the ceiling function. The number $m_s\geq 0$ (\enquote{s} for \enquote{Sobolev}) is the maximal order $|\beta|$ of multiindices $\beta$ such that $\alpha-|\beta|>0$. If $u\in H^\alpha_p(\mathbb{R}^n)$ and $|\beta|\leq m_s$, then the distributional derivative $D^\beta u$ of $u$ is an element of $H^{\alpha-|\beta|}_p(\mathbb{R}^n)$ {and} we can find 
a $(\alpha-|\beta|,p)$-quasi  continuous representative $(D^\beta u)^\sim$ of $D^\beta u$.

Let $\Gamma\subset \mathbb{R}^n$ and $|\beta|\leq m_s$. {We write $(D^\beta u)^\sim|_\Gamma=0$ whenever $(D^\beta u)^\sim=0$ $(\alpha-|\beta|,p)$-q.e.\ on $\Gamma$, noting that the latter is immediate if $\cpct_{\alpha-|\beta|,p}(\Gamma)=0$.} By $\widetilde{H}^\alpha_p(\mathbb{R}^n\setminus \Gamma)$ we denote the closure of $C_c^\infty(\mathbb{R}^n\setminus \Gamma)$ in $H^\alpha_p(\mathbb{R}^n)$. A well known theorem by Netrusov states that {for all $1<p<\infty$ and $\alpha>0$ we have}
\begin{equation}\label{E:traceqe}
\widetilde{H}^\alpha_p(\mathbb{R}^n\setminus \Gamma)=\left\lbrace u\in H^\alpha_p(\mathbb{R}^n):\ ((D^\beta u)^\sim|_\Gamma)_{|\beta|\leq m_s}=0\right\rbrace,
\end{equation}
see \cite[Corollary 10.1.2]{AH96} and \cite{HedbergNetrusov07, Netrusov93}. Recall \eqref{E:capdim}. In the special case that $0<\alpha< \alpha_p(\Gamma)+1$ or $\alpha= \alpha_p(\Gamma)+1$ and $\cpct_{\alpha_p(\Gamma),p}(\Gamma)=0$ identity \eqref{E:traceqe} is equivalent to \eqref{E:kerTr}, which stated that $\ker\mathrm{Tr}_{\Gamma,0}= \widetilde{H}^\alpha_p(\mathbb{R}^n\setminus \Gamma)$, where $\mathrm{Tr}_{\Gamma,0}$ is the operator considered in \eqref{E:traceopqe}. Earlier variants of \eqref{E:traceqe} in this case were shown in \cite{Bagby72, Deny50, Havin68, Hedberg72, Sobolev63}, and earlier variants  of \eqref{E:traceqe} for \enquote{regular} $\Gamma$ can for instance be found in \cite{Necas67, Triebel78}.

Let 
\[m_c:=\left\lceil\alpha-\frac{n}{p}\right\rceil - 1.\]
If $m_c\geq 0$, then $m_c$ (\enquote{c} for \enquote{continuous}) indicates the the maximal order $|\beta|$ up to which the $D^\beta u$ are continuous. We have $m_c\geq 0$ if and only if $\alpha>\frac{n}{p}$. If $m_c<0$, then neither $u$ nor any of its derivatives are continuous. Note that $m_c<m_s$ if and only if the interval $[\alpha-\frac{n}{p},\alpha)$ contains an integer; in this case it contains the integers $m_c+1,\ldots,m_s$. 

Now suppose that $\Gamma$ is the support of a non-zero Radon measure $\mu$ and $d:\Gamma\to [0,\infty)$ is a such that {\eqref{E:Dbar} and \eqref{E:traceconddx}} hold. Let 
\begin{equation}\label{E:mt}
m_t:=\left\lceil\alpha - \frac{(n-\inf_{x\in\Gamma} d(x))}{p}\right\rceil - 1.
\end{equation}
The number $m_t\geq 0$ (\enquote{t} for \enquote{trace}) indicates the maximal order $|\beta|$ of multiindices $\beta$ such that 
\[n-(\alpha-|\beta|)p<\inf_{x\in\Gamma} d(x).\]
For such $\beta$, {i.e. if $|\beta|\leq m_t$,} Lemma \ref{L:smooth} guarantees that all sets of zero $(\alpha-|\beta|,p)$-capacity are $\mu$-null sets.
In this case the $\mu$-equivalence class of $(D^\beta u)^\sim|_\Gamma$ is well defined, and we denote it by $[D^\beta u]_\mu$.

Recall from Remark \ref{R:upregd} (i) that $d(x)\leq n$, $x\in\Gamma$. We have that
\[ m_s - n \leq m_c \leq m_t \leq m_s. \]

For any $0\leq m\leq m_t$ we can define a linear operator $[\mathrm{Tr}_{\Gamma,m}]_\mu:H^\alpha_p(\mathbb{R}^n)\to \bigotimes_{|\beta|\leq m} L^0(\Gamma,\mu)$ by
\[[\mathrm{Tr}_{\Gamma,m}]_\mu u:=([D^\beta u]_\mu)_{|\beta|\leq m},\quad u\in H^\alpha_p(\mathbb{R}^n);\]
here $L^0(\Gamma,\mu)$ denotes the space of $\mu$-equivalence classes of Borel functions on $\Gamma$. In the special case $m=0$ we recover the operator $[\mathrm{Tr}_{\Gamma,0}]_\mu$ considered in \eqref{E:traceopmu}. The kernel of $[\mathrm{Tr}_{\Gamma,m}]_\mu$ is 
\[\ker [\mathrm{Tr}_{\Gamma,m}]_\mu=\left\lbrace u\in H^\alpha_p(\mathbb{R}^n):\ [\mathrm{Tr}_{\Gamma,m}]_\mu u=0\right\rbrace; \]
here $[\mathrm{Tr}_{\Gamma,m}]_\mu u=0$ is understood in the natural sense that $[D^\beta u]_\mu=0$ $\mu$-a.e.\ on $\Gamma$ for all $|\beta|\leq m$. 

{If the assumptions stated before formula \eqref{E:mt} hold and $0\leq m\leq m_t$,} the identity $((D^\beta u)^\sim|_\Gamma)_{|\beta|\leq m}=0$ implies the identity $[\mathrm{Tr}_{\Gamma,m}]_\mu u=0$, but in general the converse implication is not guaranteed. 

The next result applies Theorem \ref{T:main} to obtain sufficient conditions for this converse implication, and hence a characterization of 
$\ker [\mathrm{Tr}_{\Gamma,m_t}]_\mu$.  
We note that a simplified statement for the case $m_t=0$ is provided in Corollary \ref{C:variant} below.

\begin{corollary}\label{C:trace}
Let $\mu$ be a  non-zero Radon measure on $\mathbb{R}^n$, $\Gamma=\supp\mu$,  $1<p<\infty$ and $\alpha>0$. Let $d:\Gamma\to [0,\infty)$ be a Borel function such that {\eqref{E:L}, \eqref{E:fatDbar} and \eqref{E:traceconddx}} hold. 
{If $m_c<m_t$, suppose additionally that} 
\begin{equation}\label{E:densitypos}
\bar{D}^{d(x)}\mu(x)>0\quad \text{for  $\ (\alpha-\max(m_c+1,0),p)$-q.e.\ $x\in \Gamma$.}
\end{equation}
{Then the following hold:}
\begin{enumerate}
\item[(i)] 
For any $u\in H^\alpha_p(\mathbb{R}^n)$, 
{and any $0\leq m\leq m_t$,} 
\begin{equation}\label{E:muident}
[\mathrm{Tr}_{\Gamma,m}]_\mu u=0\quad \text{if and only if}\quad ((D^\beta u)^\sim|_\Gamma)_{|\beta|\leq m}=0.
\end{equation}
\item[(ii)] 
If $m_t<m_s$, suppose additionally that 
\begin{equation}
\label{E:capzero}
\cpct_{\alpha-(m_t+1),p}(\Gamma)=0.
\end{equation}
%
Then 
\begin{equation}\label{E:kerident}
\ker [\mathrm{Tr}_{\Gamma,m_t}]_\mu=\widetilde{H}^\alpha_p(\mathbb{R}^n\setminus \Gamma).
\end{equation}
\end{enumerate}
\end{corollary}

\begin{proof}
For every multiindex $\beta$ with $|\beta|\leq m_c$ the distributional derivative $D^\beta u$ has a unique continuous representative, so that the equivalence
\begin{align}
\label{E:deriv_equiv}
[D^\beta u]_\mu u=0\quad \text{if and only if} \quad ((D^\beta u)^\sim|_\Gamma)=0
\end{align}
follows similarly as Remark \ref{R:trivial}. This gives \eqref{E:muident} in the case $0\leq m\leq m_c$. If $\max(m_c+1,0)\leq m\leq m_t$, then for any $\beta$ with $\max(m_c+1,0)\leq |\beta|\leq m$ we have \eqref{E:deriv_equiv} by Theorem \ref{T:main} and since \eqref{E:densitypos} implies that $\bar{D}^{d(x)}\mu(x)>0$ for $(\alpha-|\beta|,p)$-q.e.\ $x\in \Gamma$. Here we use that a set of zero $(\sigma,p)$-capacity has zero $(\sigma',p)$-capacity for any $0<\sigma'\leq \sigma$. Together, this gives (i), and 
{by \eqref{E:traceqe}}
it gives \eqref{E:kerident} in the case that $m_t=m_s$. If $m_t\leq m_s-1$ then \eqref{E:deriv_equiv} holds for $m_t+1\leq |\beta|\leq m_s$ by the assumption that $\cpct_{\alpha-(m_t+1),p}(\Gamma)=0$, which, arguing as before, implies that $\cpct_{\alpha-|\beta|,p}(\Gamma)=0$ for all the $\beta$ in this range. {By \eqref{E:traceqe} this gives \eqref{E:kerident}.}
\end{proof}

In the special case of $d$-measures $\mu$ and resulting $d$-sets $\Gamma$ we obtain the following.
 
\begin{corollary}\label{C:dsettrace} 
Let $0<d\leq n$, $\mu$ be a $d$-measure on $\mathbb{R}^n$, $\Gamma=\supp\mu$ and $1<p<\infty$. Then \eqref{E:kerident} holds
for all $\alpha>\frac{n-d}{p}$.
\end{corollary}

\begin{proof}
In the case $m_t=m_s$ all hypotheses of Corollary \ref{C:trace} are obviously satisfied. If instead $m_t\leq m_s-1$, then we can use that, by the definition \eqref{E:mt} of $m_t$,
\begin{equation}\label{E:alphaleq}
\alpha\leq \frac{n-d}{p}+m_t+1.
\end{equation}
We claim that \eqref{E:alphaleq} implies
\begin{equation}\label{E:capzerod}
\cpct_{\alpha-(m_t+1),p}(\Gamma)=0,
\end{equation}
so that again all hypotheses of Corollary \ref{C:trace} (ii) hold. To see \eqref{E:capzerod}, assume first that the inequality in \eqref{E:alphaleq} is strict. Then $d<n-(\alpha-m_t-1)p$, and we can recycle an argument from the proof of \cite[Proposition 6.7]{CHM21}:  For any open ball $B$ we have $\mathcal{H}^d(B\cap \Gamma)<\infty$ and therefore $\mathcal{H}^{n-(\alpha-m_t-1)p}(B\cap \Gamma)=0$. Exhausting $\mathbb{R}^n$ by an increasing sequence of open balls and using the continuity of measures from below gives $\mathcal{H}^{n-(\alpha-m_t-1)p}(\Gamma)=0$, so that an application of \cite[Theorem 5.1.9]{AH96} proves \eqref{E:capzerod}. See \cite[p. 35/36]{CHM21}. Now assume that equality holds in \eqref{E:alphaleq}. Then $\mathcal{H}^{n-(\alpha-m_t-1)p}(B\cap \Gamma)=\mathcal{H}^d(B\cap \Gamma)<\infty$ for any open ball $B$. Therefore, using first \cite[Theorem 5.1.9]{AH96} and then exhausting $\mathbb{R}^n$ by open balls and applying \cite[Proposition 2.3.12]{AH96}, we can conclude \eqref{E:capzerod}. 
\end{proof}


\begin{remark}\label{R:improve} In \cite[Proposition 6.7]{CHM21} the equality \eqref{E:kerident} was shown for $d$-sets under the assumptions that $0<d<n$, $\alpha>\frac{n-d}{p}$ and $\alpha-\frac{n-d}{p}$ is not an integer, along with a similar result for Besov spaces. The exclusion of integer $\alpha-\frac{n-d}{p}$ was needed to apply an extension theorem \cite[Theorem VII.3]{JW84}. Note that the non-integer case corresponds to having $\frac{n-d}{p}+m_t<\alpha<\frac{n-d}{p}+m_t+1$ and the integer case to $\alpha=\frac{n-d}{p}+m_t+1$. Corollary \ref{C:dsettrace} is thus an extension of \cite[Proposition 6.7]{CHM21} in the case of fractional Sobolev spaces. Indeed, the question whether \eqref{E:kerident} holds for the integer case was formulated as an open problem in \cite[Remark 6.9]{CHM21}, and for the case of fractional Sobolev spaces Corollary \ref{C:dsettrace} gives an affirmative answer. For $p=2$ and $\frac{n-d}{2}<\alpha<\frac{n-d}{2}+1$ a slightly improved version of \cite[Proposition 6.7]{CHM21} was presented in \cite[Theorem 3.9]{CCh-WCGHM25} {(the main improvement being the inclusion of the case $d=n$).}  For $d=n$ and smoothness $\alpha=1$ the coincidence
\begin{equation}\label{E:kerident0}
\ker [\mathrm{Tr}_{\Gamma,0}]_\mu=\widetilde{H}^1_2(\mathbb{R}^n\setminus \Gamma)
\end{equation}
was stated as \cite[Assumption 3.12 (ii)]{CCh-WCGHM25}. For $d=n$ and $0<\alpha\leq 1$ we have $m_t=m_s=0$, so that Corollary \ref{C:dsettrace} gives \eqref{E:kerident0}. This means that \cite[Assumption 3.12 (ii)]{CCh-WCGHM25} is not needed for the validity of Corollary 3.13 and Lemma 3.15 in \cite{CCh-WCGHM25}.
\end{remark}

Recall that in \eqref{E:capdim} the number $\alpha_p(\Gamma)$ was defined to be $\alpha_p(\Gamma)=\frac{1}{p}(n-\dim_H\Gamma)$. {We provide some more detailed comments on the significance of $\alpha_p(\Gamma)$.} 

\begin{remark}\label{R:capdim} \mbox{}
\begin{enumerate}
\item[(i)] By a well-known characterization of the Hausdorff dimension in terms of Riesz capacities \cite[Theorem 8.9]{Mattila}, together with the fact that, for $\alpha>\frac{n}{p}$, points have positive $(\alpha,p)$-capacity, we have
\[\alpha_p(\Gamma)=\sup\{\alpha\geq 0:\ \cpct_{\alpha,p}(\Gamma)=0\}=\inf\{\alpha\geq 0:\ \cpct_{\alpha,p}(\Gamma)>0\},\]
for any nonempty Borel set $\Gamma\subset \mathbb{R}^n$. If $\cpct_{\alpha,p}(\Gamma)=0$, which happens in particular if $0<\alpha<\alpha_p(\Gamma)$, then, where $\widetilde{B}(\Gamma)$ is as defined in {the introduction}, $\widetilde{B}(\Gamma)=\{0\}$ and $\ker\mathrm{Tr}_{\Gamma,0}=H_p^\alpha(\mathbb{R}^n)={\widetilde{H}_p^\alpha(\mathbb{R}^n\setminus\Gamma)}$. In this case $\Gamma$ is said to be a removable set for $H_p^\alpha(\mathbb{R}^n)$, see for instance \cite[Section 2.9, p. 51]{AH96} or \cite{HKM17,HMS23} and the references cited there. If $\cpct_{\alpha,p}(\Gamma)>0$, which happens in particular if {$\alpha_p(\Gamma)<\alpha$}, then $\widetilde{H}_p^\alpha(\mathbb{R}^n\setminus \Gamma)\subset \ker\mathrm{Tr}_{\Gamma,0}\subsetneq H_p^\alpha(\mathbb{R}^n)$. {If $\alpha_p(\Gamma)<\alpha<\alpha_p(\Gamma)+1$ or $\alpha=\alpha_p(\Gamma)+1$ and $\cpct_{\alpha_p(\Gamma),p}(\Gamma)=0$,} then by \cite[Corollary 10.1.2]{AH96} also $\ker\mathrm{Tr}_{\Gamma,0}\subset \widetilde{H}_p^\alpha(\mathbb{R}^n\setminus \Gamma)$.
\item[(ii)] Let $\mu$ be a non-zero Radon measure and $\Gamma=\supp\mu$. Then $\mu(B(x,r)\cap \Gamma)>0$ for all $x\in \Gamma$ and all $r>0$. Let $d:\Gamma\to [0,\infty)$ be a Borel function satisfying {\eqref{E:Dbar}}. For $\alpha_p(\Gamma)<\alpha\leq \frac{1}{p}(n-\inf_{x\in\Gamma}d(x))$ there could be some $x\in \Gamma$ and $r>0$ such that $\cpct_{\alpha,p}(B(x,r)\cap \Gamma)=0$, see Example \ref{Ex:hybridcont} below. This is excluded by assumption \eqref{E:traceconddx}, which implies that $\cpct_{\alpha,p}(B(x,r)\cap \Gamma)>0$ for all $x\in \Gamma$ and all $r>0$ by Lemma \ref{L:smooth} (i). {If \eqref{E:L} and \eqref{E:fatDbar} hold and $S_{d,\mu}(\Gamma)=\emptyset$, then \eqref{E:dimH} gives
$\alpha_p(\Gamma)=\frac{1}{p}(n-\sup_{x\in\Gamma}d(x))$.}
\end{enumerate}
\end{remark}

We point out that in the case $m_t=0$ in Corollary \ref{C:trace} (ii) the following conditions are sufficient for the validity of \eqref{E:kerTrmu} for $\alpha$ from a certain and, in comparison with \eqref{E:Netrusov0} and \eqref{E:kerTr},
restricted range.


\begin{corollary}\label{C:variant}
Let $\mu$ be a  non-zero Radon measure on $\mathbb{R}^n$, $\Gamma=\supp\mu$ and  $1<p<\infty$. Suppose that $d:\Gamma\to[0,\infty)$ is a Borel function such that \eqref{E:L} and \eqref{E:fatDbar} hold {and 
\begin{equation}\label{E:variant}
\frac{n-\inf_{x\in\Gamma}d(x)}{p}<\alpha<\alpha_p(\Gamma)+1.
\end{equation}
\begin{enumerate}
\item[(i)] If $\alpha\leq \frac{n}{p}$ also assume that $\cpct_{\alpha,p}(S_{d,\mu}(\Gamma))=0$. Then \eqref{E:kerTrmu} holds. This remains true if $\alpha=\alpha_p(\Gamma)+1$ and $\cpct_{\alpha_p(\Gamma),p}(\Gamma)=0$. 
\item[(ii)] If $\cpct_{\min(\frac{n}{p},\alpha_p(\Gamma)+1),p}(\Gamma)=0$, then \eqref{E:kerTrmu} holds for all $\alpha$ satisfying \eqref{E:variant} and for $\alpha=\alpha_p(\Gamma)+1$.
\end{enumerate}}
\end{corollary}


\begin{proof} {
The left inequality in \eqref{E:variant} implies that $m_t=0$ and, as a consequence, $m_c\leq 0$. For $m_c=0$ the first condition in Corollary \ref{C:trace} (ii) (a) holds. By the definition \eqref{E:capdim} of $\alpha_p(\Gamma)$ we have $m_c<0$ if and only if $\alpha\leq \frac{n}{p}$, and for this case $\cpct_{\alpha,p}(S_{d,\mu}(\Gamma))=0$ is assumed. This gives the conditions in Corollary \ref{C:trace} (ii) (a).  By {Remark \ref{R:upregd} (i) and the right inequality in }\eqref{E:variant}
we have 
\begin{equation}\label{E:supsmall}
\sup_{x\in\Gamma}d(x)<n-(\alpha-1)p,
\end{equation} 
which implies that $\bar{D}^{n-(\alpha-1)p}\mu(x)=\infty$ for all $x\in\Gamma$ and, by \cite[Theorem 6.9]{Mattila}, $\mathcal{H}^{n-(\alpha-1)p}(B\cap \Gamma)=0$ for any ball $B$. We can again use \cite[Theorem 5.1.9]{AH96}, exhaustion by balls and \cite[Proposition 2.3.12]{AH96} to obtain $\cpct_{\alpha-1,p}(\Gamma)=0$, which is the second condition in Corollary \ref{C:trace} (ii) (b). Corollary \ref{C:trace} (ii) now gives \eqref{E:kerTrmu}, which proves (i). Statement (ii) is a straightforward consequence of (i).}
\end{proof}

\begin{examples}\label{Ex:hybridcont}
In Example \ref{Ex:hybrid} we have $\alpha_2(\Gamma)=1-\frac{\log 2}{\log 3}$. For $0<\alpha\leq \alpha_2(\Gamma)$ the set $\Gamma$ has zero $(\alpha,2)$-capacity, and for $\alpha_2(\Gamma)<\alpha\leq \frac12$ the set $\Gamma_2\setminus \Gamma_1$ has zero $(\alpha,2)$-capacity. If $\frac12<\alpha\leq 2-\frac{\log 2}{\log 3}$, then $B(x,r)\cap \Gamma$, for $x\in \Gamma$ and $r>0$, has positive $(\alpha,2)$-capacity and \eqref{E:kerTrmu} holds with $p=2$.
\end{examples}

Let the hypotheses of Lemma \ref{L:smooth} (i) or (ii) be satisfied. Recall from the introduction that $\widetilde{B}(\Gamma)$ denotes the space of $(\alpha,p)$-q.e.\ equivalence classes of restrictions $\widetilde{u}|_\Gamma$, $u\in H^\alpha_p(\mathbb{R}^n)$, and that it is a Banach space (Hilbert for $p=2$) if endowed with the (quotient) norm 
\begin{equation}\label{E:capnorm}
\|\varphi\|_{\widetilde{B}(\Gamma)}:=\inf\big\lbrace \|u\|_{H_p^\alpha(\mathbb{R}^n)}:\ u\in H_p^\alpha(\mathbb{R}^n),\ \mathrm{Tr}_{\Gamma,0} u=\varphi\big\rbrace,\quad \varphi\in \widetilde{B}(\Gamma).
\end{equation}
Recall also that {$[\widetilde{B}(\Gamma)]_\mu\subset L^0(\Gamma,\mu)$} denotes the space of $\mu$-equivalence classes of restrictions $\widetilde{u}|_\Gamma$, $u\in H^\alpha_p(\mathbb{R}^n)$, which is a Banach space (Hilbert for $p=2$) if endowed with the (quotient) norm 
\begin{equation}\label{E:munorm}
\|\varphi\|_{[\widetilde{B}(\Gamma)]_\mu}:=\inf\big\lbrace \|u\|_{H_p^\alpha(\mathbb{R}^n)}:\ u\in H_p^\alpha(\mathbb{R}^n),\ [\mathrm{Tr}_{\Gamma,0}]_\mu u=\varphi\big\rbrace, \quad \quad \varphi\in [\widetilde{B}(\Gamma)]_\mu.
\end{equation}

The following lemma is immediate from the above.
\begin{lemma}\label{L:allequiv}
Let the hypotheses of Lemma \ref{L:smooth} (i) or (ii) be satisfied. Let $\iota:\widetilde{B}(\Gamma)\to [\widetilde{B}(\Gamma)]_\mu$ denote the canonical linear surjection mapping an element of $\widetilde{B}(\Gamma)$ to its equivalence class in $[\widetilde{B}(\Gamma)]_\mu$. Then the following are equivalent: 
\begin{enumerate}
\item[(i)] the equivalence \eqref{E:wish} holds,  
\item[(ii)] {the map $\iota$ is injective,}
\item[(iii)]{the map $\iota$ is an isometric isomorphism. }
\end{enumerate}
\end{lemma}

{
Under slightly more restrictive assumptions, cf. Remarks \ref{R:basicimplication} and \ref{R:basicimplicationkappa}, we can observe the following.
\begin{corollary}
Let $\mu$ be a  non-zero Radon measure on $\mathbb{R}^n$, $\Gamma=\supp\mu$ and $1<p<\infty$. If
\begin{enumerate}
\item[(i)] $0<\alpha\leq \frac{n}{p}$ and $d:\Gamma\to [0,\infty)$ is a Borel function such that {\eqref{E:L}, \eqref{E:fatDbar},  \eqref{E:traceconddx} and \eqref{E:suffconddx} are satisfied or}
\item[(ii)] $\alpha=\frac{n}{p}$ and $\kappa:\Gamma\to  [0,\infty)$ is a Borel function such that {\eqref{E:Lkappa}, \eqref{E:fatDbarkappa}, \eqref{E:tracecondhx} and \eqref{E:suffcondh} are satisfied,} 
\end{enumerate}
{then (i), (ii) and (iii) in Lemma \ref{L:allequiv} hold}.
\end{corollary}
}

\begin{remark}
The first-named author would like to correct a related imprecise claim made in \cite{HR-PT2023}: In \cite[formula (5) of Corollary 2.1]{HR-PT2023} the specification \enquote{$\mu$-a.e.} must be replaced by \enquote{q.e.} and \cite[Proposition 2.1]{HR-PT2023} should be disregarded. With these corrections all other results remain valid.
\end{remark}

\section{Applications to domains}\label{S:domains}

We discuss applications of our results to domains, i.e., non-empty open sets $\Omega\subset \mathbb{R}^n$; {our results in this section hold for all $n\geq 1$ except where indicated otherwise.}

{Our first result is a straightforward application of Theorem \ref{T:main}, in a case where \eqref{E:suffconddx} holds with $S_{d,\mu}(\Gamma)$ not necessarily empty.}

{\begin{theorem} \label{T:domain0}
Suppose $1<p<\infty$, $\alpha>0$, and $\Omega \subset \R^n$ is a domain such that $\mathcal{L}^n(\partial \Omega)=\mathrm{Cap}_{\alpha,p}(S) = 0$, where 
\begin{equation} \label{E:S}
S := \big\lbrace x\in \partial \Omega: \limsup_{r\to 0} \frac{\mathcal{L}^n(B(x,r)\cap \Omega)}{r^n}=0\big\rbrace. 
\end{equation}
Then 
\begin{equation}\label{E:domain}
\text{$u=0\ $ $\mathcal{L}^n$-a.e.\ on $\Omega$}\quad\text{if and only if}\quad\text{$\widetilde{u}=0$ $(\alpha,p)$-q.e.\ on $\overline{\Omega}$,} \qquad u\in H^\alpha_p(\R^n),
\end{equation}
where $\widetilde{u}$ denotes an arbitrary $(\alpha,p)$-quasi continuous representative of $u$.
\end{theorem}
\begin{proof}
Let $\Gamma := \overline{\Omega}$, $\mu := \mathcal{L}^n|_\Gamma$ and $d(x):= n$, $x\in \Gamma$.  Then, if $\alpha p \leq n$, $\Gamma$, $\mu$ and $d$ satisfy the conditions of Theorem \ref{T:main} with $S_{d,\mu}(\Gamma)=S$, and \eqref{E:domain} follows from Theorem \ref{T:main}, since $\mathcal{L}^n(\partial \Omega)=\mathrm{Cap}_{\alpha,p}(S) = 0$. If $\alpha p>n$ then \eqref{E:domain} follows from Remark \ref{R:trivial}.   
\end{proof}
}

{If $\alpha p\leq n$ there exist non-empty sets $S\subset \partial \Omega$ with $\mathrm{Cap}_{\alpha,p}(S) = 0$. In particular $\mathrm{Cap}_{\alpha,p}(S) = 0$ if $S$ is countable with only finitely many limit points, as a consequence of \cite[Theorems 5.1.2-5.1.4]{AH96}. This choice for $S$ allows, for example, $\partial \Omega$ to have infinitely many outward cusps. }

{ 
Recall from \cite{Hajlaszetal08} that a domain $\Omega\subset \mathbb{R}^n$ satisfies the \textit{{measure density condition}} if there is some $c>0$ such that, for all $x\in \Omega$ and $0<r\leq 1$, 
\begin{equation}\label{E:measuredenscond}
\mathcal{L}^n(B(x,r)\cap \Omega)\geq c\:r^n.
\end{equation}

\begin{remark}\label{R:measuredenscond}
Domains satisfying the measure density condition \eqref{E:measuredenscond}
are called ``$n$-thick domains'' in \cite{Rychkov00}, ``interior regular domains'' in \cite{CHM21}, ``open $n$-sets'' in \cite{CCh-WH25}, and simply \enquote{$n$-sets} in \cite[p. 205]{JW84} or \cite{FarkasJacob2001}. (We note that in our definition \eqref{E:dset} we require an $n$-set to be closed.) 
\end{remark}}

{Whether or not $\alpha p\leq n$, we have the following corollary of Theorem \ref{T:domain0} which is,} in fact, a variant of the $n$-set case in {Corollary \ref{C:dset0}.}

\begin{corollary}\label{C:domain}
Let $\Omega\subset \mathbb{R}^n$ be a domain satisfying the measure density condition. Then $\overline{\Omega}$ is an $n$-set, {$\mathcal{L}^n(\partial\Omega)=0$ and $S=\emptyset$, where $S$ is given by \eqref{E:S}. Moreover,} for any $1<p<\infty$ and $\alpha>0$, {the equivalence \eqref{E:domain} holds,}
where $\widetilde{u}$ denotes an arbitrary $(\alpha,p)$-quasi continuous representative of $u$.
\end{corollary}

\begin{proof} {The measure density condition \eqref{E:measuredenscond} implies that $\overline{\Omega}$ is an $n$-set and $\mathcal{L}^n(\partial\Omega)=0$, see \cite[Proposition 1, p. 205]{JW84} for a proof. In particular, \eqref{E:measuredenscond} holds, for some $c>0$, for all $x\in \overline{\Omega}$, so that $S=\emptyset$. Thus \eqref{E:domain} follows from Theorem \ref{T:domain0}.}
\end{proof}

Given a domain $\Omega\subset \mathbb{R}^n$, $1\leq p<\infty$ and a positive integer $m$, we write $W^m_p(\Omega)$ for the Sobolev space of all $u\in L^q(\Omega)$ with $D^\beta u\in L^q(\Omega)$, $|\beta|\leq m$, endowed with the intrinsic norm $u\mapsto \sum_{|\beta|\leq m}\|D^\beta u\|_{L^q(\Omega)}$. We say that a domain $\Omega$ has the $W^m_p$-extension property if there is a bounded linear extension operator $\mathrm{E}_\Omega:W_p^m(\Omega)\to W^m_p(\R^n)$ with $(\mathrm{E}_\Omega u)|_\Omega=u$, $u\in W^m_p(\Omega)$. For $p=2$ and arbitrary positive integer $m$ every domain $\Omega$ for which the restriction $u\mapsto u|_\Omega$ is surjective from $W^m_2(\mathbb{R}^n)$ onto $W_2^m(\Omega)$ has the $W^m_2$-extension property, cf. \cite[p. 1221]{Hajlaszetal08}.

For $1<p<\infty$ and $\alpha>0$ we define $H_p^\alpha(\Omega)$ as the space of all $u\in L^p(\Omega)$ such that there is some $v\in H^\alpha_p(\mathbb{R}^n)$ with $u=v|_\Omega$, endowed with the quotient norm. We say that $\Omega$ has the $H^\alpha_p$-extension property if there is a bounded linear extension operator $\mathrm{E}_\Omega:H^\alpha_p(\Omega)\to H^\alpha_p(\R^n)$ with $(\mathrm{E}_\Omega u)|_\Omega=u$, $u\in H^\alpha_p(\Omega)$.  For $p=2$ and arbitrary $\alpha>0$ every domain $\Omega$  trivially has the $H_2^\alpha(\Omega)$-extension property \cite[p.~142]{Rychkov00}.

\begin{lemma}\label{L:extensiondomain} Let $\Omega\subset \mathbb{R}^n$ be a domain.
\begin{enumerate}
\item[(i)] If, for some $1\leq p<\infty$ and positive integer $m$, the restriction $u\mapsto u|_\Omega$ is surjective  from $W^m_p(\mathbb{R}^n)$ onto $W_p^m(\Omega)$, then $\Omega$ satisfies the measure density condition. This is true in particular if $\Omega$ has the $W^m_p$-extension property.
\item[(ii)] If $\Omega$  satisfies the measure density condition, then it has the $H^\alpha_p$-extension property for every $1<p<\infty$ and $\alpha>0$.
\item[(iii)]  If $1<p<\infty$, $m$ is a positive integer and $\Omega$ has the $W^m_p$-extension property, then $H^m_p(\Omega)$ equals $W^m_p(\Omega)$ in the sense of equivalently normed spaces.
\item[(iv)] If {$n\geq 2$ and} $\Omega$ is an $(\varepsilon,\delta)$-domain, then $\Omega$ has the $W^m_p$-extension property for every $1\leq p<\infty$ and every positive integer $m$.
\end{enumerate}
\end{lemma}

\begin{proof} In \cite[Theorem 2]{Hajlaszetal08} it was shown that with the assumptions in (i), the domain $\Omega$ satisfies the measure density condition. Item (ii) was proved in \cite[Theorem 1.1(a)]{Rychkov00}. Item (iii) is {well known} and easy to see. Item (iv) was proved in \cite[Theorem 1]{Jones81}.
\end{proof}

\begin{remark}\mbox{}
\begin{enumerate}
\item[(i)] For $1<p<\infty$ the surjectivity of the restriction in (i) and the $W^m_p$-extension property are actually equivalent \cite[Theorem 5]{Hajlaszetal08}. 
\item[(ii)] It is well known that every Lipschitz domain $\Omega\subset\mathbb{R}^n$, with $n\geq 2$, is an $(\varepsilon,\delta)$-domain \cite{Jones81}. Well-known extension theorems for Lipschitz domains are \cite[Chapter VI, Theorem 5]{Stein70}, \cite[Corollary 1.6]{Strichartz67} and \cite[Theorem 1]{Kalyabin85}. Further extension results for  $(\varepsilon,\delta)$-domains are \cite[Theorem 8]{Rogers06} and \cite[Theorem 2]{Seeger89}. 
\item[(iii)] Related extension results for fractional Sobolev spaces on domains defined using intrinsic norms and in the case that $0<\alpha<1$ can be found in \cite{Zhou15}. Extension results for Haj\l asz-Triebel-Lizorkin spaces are provided in \cite{HIT16}.
\end{enumerate}
\end{remark}

\begin{remark} For $1<p<\infty$, {$\alpha=1$,} and domains $\Omega$ having the $W^1_p$-extension property, the nontrivial implication in \eqref{E:domain} was proved in \cite[Theorem 6.1 and Remark 6.2]{Biegert09}; see also \cite[Theorems 3.20, 3.23 and 3.27]{Biegert09b}. For $p=2$,  $0<\alpha\leq 1$ and domains $\Omega$ having the $W^1_p$-extension property the nontrivial implication in \eqref{E:domain} can be shown using Dirichlet form methods, see \cite{ArendtWarma03},\cite[p. 53]{FOT2011} and the arguments used to prove \cite[Theorem 4.4]{FarkasJacob2001}. For smooth domains this kind of argument works for $p=2$ and $0<\alpha<\frac{3}{2}$ by combining interpolation \cite[Theorem 4.1]{Seeley72} and using suitably defined capacities \cite{FukushimaKaneko85}. {Theorem \ref{T:domain0}, Corollary \ref{C:domain}} and Lemma \ref{L:extensiondomain} extend these results.
\end{remark}

Given a domain $\Omega$ having the $H^\alpha_p$-extension property and a closed subset $\Gamma$ of $\overline{\Omega}$, we consider the composition 
\[\mathrm{Tr}_{\Gamma,0}^\Omega:=\mathrm{Tr}_{\Gamma,0}\circ \mathrm{E}_\Omega.\]
For $\alpha=1$ this kind of trace operator was used, for example, in \cite{CHR-PT24} and \cite{HR-PT2021}. By $H^\alpha_{p,0}(\overline{\Omega}\setminus \Gamma)$ we denote the closure of $\{v|_\Omega:\ v\in C_c^\infty(\mathbb{R}^n\setminus \Gamma)\}$ in $H_p^\alpha(\Omega)$. In the special case $\Gamma=\partial\Omega$ we recover the familiar definition of $H^\alpha_{p,0}(\Omega)$ as the closure of $C_c^\infty(\Omega)$ in $H_p^\alpha(\Omega)$. In the following, $\widetilde{B}(\Gamma)$ and $[\widetilde{B}(\Gamma)]_\mu$ are as defined in Sections \ref{S:intro} and \ref{S:kernels}. Our next 
lemma extends \cite[Corollary 6.3]{Biegert09}, where it was shown for $\alpha=1$ {in the case that $\Omega$ has the $W^1_p$-extension property. Note that if $\Omega$ has the $W^1_p$-extension property then it satisfies the measure density condition by Lemma \ref{L:extensiondomain} (i). Further, if $\Omega$ satisfies the measure density condition then $\Omega$ has the $H^\alpha_p$-extension property by Lemma \ref{L:extensiondomain} (ii), and $\mathcal{L}^n(\partial \Omega)=\mathrm{Cap}_{\alpha,p}(S) = 0$ by Corollary \ref{C:domain}, so that all the conditions of the next lemma are satisfied.}

\begin{lemma}\label{L:traceop}
Let $1<p<\infty$ and $\alpha>0$. {Suppose that $\Omega\subset \mathbb{R}^n$ has the $H^\alpha_p$-extension property and that $\mathcal{L}^n(\partial \Omega)=\mathrm{Cap}_{\alpha,p}(S) = 0$, where $S$ is given by \eqref{E:S}.} 
Let $\Gamma$ be a closed subset of $\overline{\Omega}$ such that $\overline{\Omega}\setminus \Gamma$ is dense in $\overline{\Omega}$.
\begin{enumerate}
\item[(i)] The operator $\mathrm{Tr}_{\Gamma,0}^\Omega:H^\alpha_p(\Omega)\to \widetilde{B}(\Gamma)$ is bounded, linear and surjective. It does not depend on the 
particular choice of the extension operator $\mathrm{E}_\Omega$, {indeed if $u\in H^\alpha_p(\R^n)$ then $\mathrm{Tr}_{\Gamma,0}^\Omega (u|_\Omega) = \mathrm{Tr}_{\Gamma,0}u$.}
\item[(ii)] If $0<\alpha<\alpha_p(\Gamma)+1$ or $\alpha=\alpha_p(\Gamma)+1$ and $\cpct_{\alpha_p(\Gamma),p}(\Gamma)=0$, then $\ker \mathrm{Tr}_{\Gamma,0}^\Omega=H^\alpha_{p,0}(\overline{\Omega}\setminus\Gamma)$. {In particular, choosing $\Gamma=\partial\Omega$ gives $\ker \mathrm{Tr}_{\partial\Omega,0}^\Omega=H^\alpha_{p,0}(\Omega)$.}
\end{enumerate}
\end{lemma}

\begin{proof}
The linearity and boundedness of $\mathrm{Tr}_{\Gamma,0}^\Omega$ are clear. To prove the surjectivity, suppose that $f\in \widetilde{B}(\Gamma)$ and $u\in H^\alpha_p(\mathbb{R}^n)$ is such that $f=\mathrm{Tr}_{\Gamma,0} u$. Then $u|_\Omega\in H^\alpha_p(\Omega)$. For any extension operator $\mathrm{E}_\Omega$, since both $u$ itself and $\mathrm{E}_\Omega (u|_\Omega)$ extend $u|_\Omega$ to all of $\mathbb{R}^n$, we have 
$(\mathrm{E}_\Omega (u|_\Omega))^\sim=\widetilde{u}$ $(\alpha,p)$-q.e.\ on $\overline{\Omega}$ by Theorem \ref{T:domain0}. 
In particular, 
$\mathrm{Tr}_{\Gamma,0}^\Omega {(u|_\Omega)}=\mathrm{Tr}_{\Gamma,0}(\mathrm{E}_\Omega (u|_\Omega))=\mathrm{Tr}_{\Gamma,0} u=f$. This gives (i).

To prove (ii), suppose first that $u\in H_p^\alpha(\Omega)$ is such that $\mathrm{Tr}_{\Gamma,0}^\Omega=0$ in $\widetilde{B}(\Gamma)$, that is, $(\mathrm{E}_\Omega u)^\sim|_\Gamma=0$ $(\alpha,p)$-q.e. Then $\mathrm{E}_\Omega u\in \widetilde{H}_p^\alpha(\mathbb{R}^n\setminus\Gamma)$ by \eqref{E:kerTr}, so that there is a sequence $(v_k)\subset C_c^\infty(\mathbb{R}^n\setminus \Gamma)$ converging to $\mathrm{E}_\Omega u$ in $H^\alpha_p(\mathbb{R}^n)$. Consequently also 
\begin{equation}\label{E:convinOm}
\lim_k v_k|_\Omega=u\quad \text{in $H^\alpha_p(\Omega)$},
\end{equation}
hence $u\in H_{p,0}^\alpha(\overline{\Omega}\setminus \Gamma)$. This shows that $\ker \mathrm{Tr}_{\Gamma,0}^\Omega\subset H^\alpha_{p,0}(\overline{\Omega}\setminus \Gamma)$. To see the converse inclusion $H^\alpha_{p,0}(\overline{\Omega}\setminus \Gamma)\subset \ker \mathrm{Tr}_{\Gamma,0}^\Omega$, suppose that $u\in H_{p,0}^\alpha(\overline{\Omega}\setminus \Gamma)$ and let $(v_k)\subset C_c^\infty(\mathbb{R}^n\setminus \Gamma)$ be such that \eqref{E:convinOm} holds. For any $k$, both $\mathrm{E}_\Omega (v_k|_\Omega)$ and $v_k$ itself are in $H_p^\alpha(\mathbb{R}^n)$, and therefore {Theorem \ref{T:domain0}}  implies that $(\mathrm{E}_\Omega v_k)^\sim=\widetilde{v}_k=0$ $(\alpha,p)$-q.e.\ on $\Gamma$. 
On the other hand, $\lim_k \mathrm{E}_\Omega v_k=\mathrm{E}_\Omega u$ in $H^\alpha_p(\mathbb{R}^n)$, because $\mathrm{E}_\Omega$ is linear and bounded. This implies that there is a sequence $(k_\ell)$ with $k_\ell\to \infty$ as $\ell\to\infty$ such that $(\mathrm{E}_\Omega u)^\sim=\lim_\ell (\mathrm{E}_\Omega v_{k_\ell})^\sim=0$ $(\alpha,p)$-q.e.\ on $\Gamma$, see for instance \cite[Proposition 2.3.8]{AH96}. Consequently $\mathrm{Tr}_{\Gamma,0}^\Omega u=0$, as desired.
\end{proof}

Similar arguments give the following consequence of \eqref{E:traceqe}. 
%

\begin{corollary}
{Let $1<p<\infty$ and $\alpha>0$ and let $\Omega\subset \mathbb{R}^n$ be as in Lemma \ref{L:traceop}.} Let $\Gamma$ be a closed subset of $\overline{\Omega}$ such that $\overline{\Omega}\setminus \Gamma$ is dense in $\overline{\Omega}$. Then
\begin{equation}\label{E:traceqedomain}
H^\alpha_{p,0}(\overline{\Omega}\setminus \Gamma)=\{u\in H^\alpha_p(\Omega):\ \big(({D^\beta \mathrm{E}_\Omega u})^\sim|_{\Gamma}\big)_{|\beta|\leq m_s}=0\}.
\end{equation}
\end{corollary}

Given a domain $\Omega$ having the $H^\alpha_p$-extension property and a non-zero Radon measure $\mu$ with $\Gamma=\supp\mu
\subset \overline{\Omega}$ and satisfying the hypotheses of Lemma \ref{L:smooth} (i), we can consider the composition 
\[[\mathrm{Tr}_{\Gamma,0}^\Omega]_\mu:=[\mathrm{Tr}_{\Gamma,0}]_\mu\circ \mathrm{E}_\Omega.\] 
The next corollary gives new results for general measures $\mu$. It is an immediate consequence of {Corollary \ref{C:variant}}.

\begin{corollary}\label{C:traceop}
{Let $1<p<\infty$ and $\alpha>0$ and let $\Omega\subset \mathbb{R}^n$ be as in Lemma \ref{L:traceop}.} Let $\mu$ be a non-zero Radon measure with support $\Gamma=\supp\mu$ contained in $\overline{\Omega}$ and such that $\overline{\Omega}\setminus \Gamma$ is dense in $\overline{\Omega}$. {Suppose that $d:\Gamma\to [0,\infty)$ is a Borel function such that \eqref{E:Dbar} and \eqref{E:traceconddx} hold.}
\begin{enumerate}
\item[(i)]  The operator $[\mathrm{Tr}_{\Gamma,0}^\Omega]_\mu:H^\alpha_p(\Omega)\to [\widetilde{B}(\Gamma)]_\mu$ is bounded, linear and surjective and does not depend on the particular choice of the extension operator $\mathrm{E}_\Omega$. 
\item[(ii)] {Suppose that  \eqref{E:L}, \eqref{E:fatDbar} and \eqref{E:variant} hold and if $\alpha\leq \frac{n}{p}$ also $\cpct_{\alpha,p}(S_{d,\mu}(\Gamma))=0$. Then 
$\ker  [\mathrm{Tr}_{\Gamma,0}^\Omega]_\mu=H^\alpha_{p,0}(\overline{\Omega}\setminus \Gamma)$. This remains true if $\alpha=\alpha_p(\Gamma)+1$ and $\cpct_{\alpha_p(\Gamma),p}(\Gamma)=0$. Choosing $\Gamma=\partial\Omega$ gives $\ker [\mathrm{Tr}_{\partial\Omega,0}^\Omega]_\mu=H^\alpha_{p,0}(\Omega)$.}
\end{enumerate}
\end{corollary}

\begin{remark} In the special case that $\Omega$ is a bounded $(\varepsilon,\delta)$-domain {in $\mathbb{R}^n$, $n\geq 2$}, $\Gamma=\partial\Omega$ a $d$-set with $n-1\leq d<n$, $\mu=\mathcal{H}^d|_{\partial\Omega}$, $1<p<\infty$ and $\frac{n-d}{p}<\alpha\leq 1$ the equality $\ker  [\mathrm{Tr}_{\partial\Omega,0}^\Omega]_\mu=H^\alpha_{p,0}(\Omega)$ was shown in  \cite[Theorem 3.5]{FarkasJacob2001}. Corollary \ref{C:traceop} extends this result to the natural parameter range conjectured in \cite[Remark 3.7]{FarkasJacob2001} and beyond the $d$-set case. 
\end{remark}

In the special case of a $d$-measure $\mu$ we obtain the following counterpart of Corollary \ref{C:dsettrace}.

\begin{corollary}\label{C:dsetdomain} 
{Let $1<p<\infty$, $0<d\leq n$ and $\alpha>\frac{n-d}{p}$, let $\Omega\subset \mathbb{R}^n$ be as in Lemma \ref{L:traceop} and let $\mu$ be a $d$-measure with support $\Gamma=\supp\mu$ contained in $\overline{\Omega}$ and such that $\overline{\Omega}\setminus \Gamma$ is dense in $\overline{\Omega}$.} Then 
\begin{equation}\label{E:keridentdomain}
\ker [\mathrm{Tr}_{\Gamma,m_t}^\Omega]_\mu=H^\alpha_{p,0}(\mathbb{R}^n\setminus \Gamma).
\end{equation}
holds, where 
\[[\mathrm{Tr}_{\Gamma,m}^\Omega]_\mu u:=([{D^\beta \mathrm{E}_\Omega u}]_\mu)_{|\beta|\leq m},\quad u\in H^\alpha_p(\Omega).\]
\end{corollary}

\begin{remark}\mbox{}
\begin{enumerate}
\item[(i)] For $\Gamma=\partial\Omega$ and positive integer $\alpha$ the identity \eqref{E:keridentdomain} was proved in \cite[Theorem 3]{Wallin91}.
\item[(ii)] Based on Corollary \ref{C:trace} one can also formulate a variant of Corollary \ref{C:dsetdomain} for more general measures $\mu$. We leave this to the reader.
\end{enumerate}
\end{remark}

\section{The convergence of Galerkin schemes}
\label{S:Galerkin}

We briefly comment on an application; for further detail see \cite{Ch-WCHR-PS2025}. Suppose that {$n\geq 1$, $p=2$, $\Gamma\subset \mathbb{R}^n$ is compact and $\alpha$ is as in \eqref{E:Netrusov0}}. Let $\mathbb{H}^\alpha(\mathbb{R}^n\setminus \Gamma)$ denote the orthogonal complement of $\ker \mathrm{Tr}_{\Gamma,0}=\widetilde{H}_2^\alpha(\mathbb{R}^n\setminus \Gamma)$ in $H_2^\alpha(\mathbb{R}^n)$ and let $(\mathbb{H}^\alpha(\mathbb{R}^n\setminus \Gamma))^\ast$ denote its topological dual. By restriction and trivial extension the space $(\mathbb{H}^\alpha(\mathbb{R}^n\setminus \Gamma))^\ast$ can be identified with the space of all elements $f$ of $H^{-\alpha}_2(\mathbb{R}^n)$ with support in $\Gamma$, e.g., \cite[Cor.~3.4]{Ch-WHM2017}. 

{ 
Let $\mu$ be a non-zero Radon measure with $\supp\mu=\Gamma$, obviously $\mu$ is finite. By $L^2(\Gamma,\mu)$ we denote the Hilbert space of all $\mu$-square integrable elements of $L^0(\Gamma,\mu)$. If $\alpha\leq \frac{n}{2}$, assume that \eqref{E:fatDbar} and \eqref{E:traceconddx} hold. Then we can, {arguing} as in Remark  \ref{R:upregd} (i), find some $d_{\min}>n-2\alpha$ such that $\mu$ is upper $d_{\min}$-regular. In this case $[\mathrm{Tr}_{\Gamma,0}]_\mu:H_2^\alpha(\mathbb{R}^n)\to L^2(\Gamma,\mu)$ is a bounded linear operator, see \cite[Theorem~1]{J79} or \cite[Theorem 7.2.2]{AH96}. If $\alpha>\frac{n}{2}$, then the boundedness of this linear operator is a consequence of the finiteness of $\mu$ and the Sobolev embedding of $H_2^\alpha(\mathbb{R}^n)$ into a space of H\"older continuous functions.
The restriction \[[\mathrm{tr_\Gamma}]_\mu:=[\mathrm{Tr}_{\Gamma,0}]_\mu|_{\mathbb{H}^\alpha(\mathbb{R}^n\setminus \Gamma)}\] 
of $[\mathrm{Tr}_{\Gamma,0}]_\mu$ to  $\mathbb{H}^\alpha(\mathbb{R}^n\setminus \Gamma)$ has the adjoint $[\mathrm{tr_\Gamma}]_\mu^\ast:L^2(\Gamma,\mu)\to (\mathbb{H}^\alpha(\mathbb{R}^n\setminus \Gamma))^\ast$.  The next corollary is a consequence of Remark \ref{R:trivial} and Lemma \ref{L:allequiv}.}

\begin{corollary}\label{C:denserange} {
Let  $p=2$, let $\Gamma$ be compact and let $\alpha$ be as in \eqref{E:Netrusov0}. If $\alpha\leq \frac{n}{2}$, assume that \eqref{E:fatDbar} and \eqref{E:traceconddx}
hold. }Then the following are equivalent:
\begin{enumerate}
\item[(i)] the equivalence \eqref{E:wish} holds, 
\item[(ii)] the operator $[\mathrm{tr_\Gamma}]_\mu$ is injective,
\item[(iii)] the operator $[\mathrm{tr_\Gamma}]_\mu^\ast$ has dense image.
\end{enumerate}
{If $\alpha >\frac{n}{2}$ or \eqref{E:L} and \eqref{E:suffconddx} hold, then (i), (ii) and (iii) are true.}
\end{corollary}

Item (iii) in Corollary \ref{C:denserange} ensures the convergence of Galerkin schemes for integral equation formulations of time-harmonic 
acoustic scattering problems  for  sound soft fractal obstacles \cite{Ch-WCHR-PS2025,CCh-WCGHM25, CCh-WGHM24, Ch-WHM2017, Ch-WHMB2021}. Formulated in terms of functional analysis, the task, given some non-empty compact $\Gamma\subset \R^n$, is to find an element $\phi$ of $(\mathbb{H}^\alpha(\mathbb{R}^n\setminus \Gamma))^\ast$ such that 
\begin{equation}\label{E:IE}
A\phi=g,
\end{equation}
where $A:(\mathbb{H}^\alpha(\mathbb{R}^n\setminus \Gamma))^\ast\to \mathbb{H}^\alpha(\mathbb{R}^n\setminus \Gamma)$ is a given invertible linear operator which is the sum of a coercive and a compact operator and $g\in \mathbb{H}^\alpha(\mathbb{R}^n\setminus \Gamma)$ is a given function. Now it is desirable to find a suitable sequence $(W_k)$ of finite dimensional subspaces $W_k$ of $(\mathbb{H}^\alpha(\mathbb{R}^n\setminus \Gamma))^\ast$ and a sequence $(\phi_k)$ of elements $\phi_k\in W_k$ converging to $\phi$ in $(\mathbb{H}^\alpha(\mathbb{R}^n\setminus \Gamma))^\ast$. Since 
\[ \left\langle[\mathrm{tr_\Gamma}]_\mu^\ast f,u\right\rangle_{(\mathbb{H}^\alpha(\mathbb{R}^n\setminus \Gamma))^\ast \times \mathbb{H}^\alpha(\mathbb{R}^n\setminus \Gamma)}=\left( f, [\mathrm{tr_\Gamma}]_\mu u\right)_{L^2(\Gamma,\mu)},\qquad f\in L^2(\Gamma,\mu),\ u\in \mathbb{H}^\alpha(\mathbb{R}^n\setminus \Gamma),\]
it seems convenient to construct a suitable sequence $(V_k)$ of finite dimensional subspaces in $L^2(\Gamma,\mu)$ and to try to \enquote{transfer} them to $(\mathbb{H}^\alpha(\mathbb{R}^n\setminus \Gamma))^\ast$ using $[\mathrm{tr_\Gamma}]_\mu^\ast$.
If $(\pi_k)$ is a sequence of finer and finer partitions $\pi_k$ of $\Gamma$ in a $\mu$-a.e.\ sense, the sequence $(V_k)$ of spaces $V_k$ of {$\pi_k$-locally constant} functions is a good choice.  (Indeed, in the case that $\Gamma$ is the boundary of a bounded Lipschitz domain, this choice for $(V_k)$ is precisely a standard piecewise-constant boundary element method; see \cite[Remark~4.2]{CCh-WCGHM25}, \cite[Remark~3.4]{Ch-WCHR-PS2025}.)
If the mesh sizes of the $\pi_k$ go to zero as $k\to\infty$, then $(V_k)$ is \emph{asymptotically dense in $L^2(\Gamma,\mu)$}, meaning that
\[\lim_{k\to\infty} \min_{f_k\in V_k}\|f_k-f\|_{L^2(\Gamma,\mu)}=0,\quad f\in L^2(\Gamma,\mu).\]
The key observation is the following immediate consequence of Corollary \ref{C:trace}. The last sentence of this corollary is part of standard theory for the Galerkin method in the case that $A$ is invertible and the sum of a coercive and a compact operator; for details see \cite[Corollary~4.2]{Ch-WCHR-PS2025}.

\begin{corollary}\label{C:asympdense} {
Let  $p=2$, let $\Gamma$ be compact and let $\alpha$ be as in \eqref{E:Netrusov0}. If $\alpha\leq \frac{n}{2}$, assume that  \eqref{E:fatDbar} and \eqref{E:traceconddx} hold.} Let $(V_k)$ be asymptotically dense in $L^2(\Gamma,\mu)$. Then the sequence $(W_k)$ of spaces $W_k:=[\mathrm{tr_\Gamma}]_\mu^\ast(V_k)$ is asymptotically dense in $(\mathbb{H}^\alpha(\mathbb{R}^n\setminus \Gamma))^\ast$ if and only if one (and hence all) of the conditions (i), (ii) or (iii) in Corollary \ref{C:denserange} holds. The Galerkin scheme for \eqref{E:IE}, based on the sequence $(W_k)$, converges for every $g\in \mathbb{H}^\alpha(\mathbb{R}^n\setminus \Gamma)$ if and only if $(W_k)$ is asymptotically dense, so if and only if one of the conditions (i), (ii) or (iii) in Corollary \ref{C:denserange} holds.
\end{corollary}

We comment on a related fact. By definition $[\mathrm{tr_\Gamma}]_\mu:\mathbb{H}^\alpha(\mathbb{R}^n\setminus \Gamma)\to [\widetilde{B}(\Gamma)]_\mu$ is surjective, and we use $[\mathrm{tr_\Gamma}]_{\mu,s}$ as a shortcut notation for this operator. Its adjoint is $([\mathrm{tr_\Gamma}]_{\mu,s})^\ast: ([\widetilde{B}(\Gamma)]_\mu)^\ast\to (\mathbb{H}^\alpha(\mathbb{R}^n\setminus \Gamma))^\ast$. We state another direct consequence of Corollary \ref{C:trace} (cf.~\cite[Proposition~3.6]{Ch-WCHR-PS2025}).

\begin{corollary}\label{C:isometries}
{Let  $p=2$, let $\Gamma$ be compact and let $\alpha$ be as in \eqref{E:Netrusov0}. If $\alpha\leq \frac{n}{2}$, assume that  \eqref{E:fatDbar} and \eqref{E:traceconddx} hold.} Then the following are equivalent: 
\begin{enumerate}
\item[(i)] the equivalence \eqref{E:wish} holds, 
\item[(ii)] the operator $[\mathrm{tr_\Gamma}]_{\mu,s}$ is an isometric isomorphism,
\item[(iii)] the operator $([\mathrm{tr_\Gamma}]_{\mu,s})^\ast$ is an isometric isomorphism.
\end{enumerate}
{If $\alpha >\frac{n}{2}$ or \eqref{E:L} and \eqref{E:suffconddx} hold, then (i), (ii) and (iii) are true.}
\end{corollary}

\section{Local isoperimetric inequalities}\label{S:iso}

As preparation for the proof of Theorems \ref{T:main} and \ref{T:main_logcase}, in this section
we quote and refine an \enquote{isoperimetric} inequality from \cite[Section 8.5]{Mazya85} which relates measures and capacities of Borel sets, see \eqref{E:mazyaoutput} below. Let $\Phi:[0,\infty)\to [0,\infty)$ be an increasing function such that $t\mapsto t\Phi(t^{-1})$ decreases at infinity with limit $\lim_{t\to \infty}t\Phi(t^{-1})=0$, so that, in particular, $\Phi$ is continuous at zero with $\Phi(0)=0$. {Suppose $n\geq 1$, $1<p<\infty$, $0<\alpha p\leq n$ and let} 
\begin{equation}\label{E:capPsi}
\Psi(t):=\begin{cases}
(t\Phi(t^{-1}))^{p'-1}\ &\text{if $\alpha p<n$},\\
t(\Phi(t^{1-p}))^{p'-1}\ &\text{if $\alpha p=n$}.
\end{cases} 
\end{equation}

In the case $\alpha p<n$ we assume that there is a constant {$c(\Phi)>0$} such that 
\begin{equation}\label{E:intcondition}
\int_u^\infty \Psi(t)\frac{dt}{t}\leq {c(\Phi)}\:\Psi(u)
\end{equation}
for all $u>0$. In the case $\alpha p=n$ we assume that for any $u_0>0$ there is a constant {$c(\Phi)=c(\Phi, u_0)>0$} such that \eqref{E:intcondition} holds for all $u\geq u_0$. 

{
\begin{examples}\label{Ex:keyex}\mbox{}
\begin{enumerate}
\item[(i)] If $\alpha p\leq n$ and $1<\beta<\infty$, then $\Phi_\beta(t):=t^{\beta}$, $t\geq 0$, satisfies the above conditions. In particular, 
 \eqref{E:intcondition} holds for $\Psi(t)=t^{-\frac{\beta-1}{p-1}}$ with $c(\Phi_\beta):=\frac{p-1}{\beta-1}$ in the case $\alpha p<n$, while in the case $\alpha p=n$ it holds for $\Psi(t)=t^{1-\beta}$ with $c(\Phi_\beta)=\frac{1}{\beta-1}$. Moreover, if $1<\beta_-<\infty$, then \eqref{E:intcondition} holds simultaneously for all $\Phi_\beta$, $\beta_-\leq \beta<\infty$, with the same constant $c(\Phi_{\beta_-})$.
\item[(ii)] For $\alpha p=n$ and $0<d<\infty$ also the function $\Phi(t):=\exp(-dt^{1-p'})$, $t>0$, $\Phi(0):=0$, satisfies the above conditions. In particular, \eqref{E:intcondition} holds for $\Psi(t)=t\exp\big(-d(p'-1)t\big)$ with $c(\Phi, u_0):=\frac{p-1}{d u_0}$.
\end{enumerate}
\end{examples}

We use the following refined variant of \cite[Section 8.5,  Corollary 2 and Remark 2]{Mazya85}. 

\begin{lemma}\label{L:Mazya}
Let $\nu$ be a finite Borel measure on $\mathbb{R}^n$, $1<p<\infty$ and $0<\alpha\leq \frac{n}{p}$. Let $\Phi$ be as specified above and suppose that  there are constants $c_0>0$ and $r_0>0$ such that 
\begin{equation}\label{E:mazyainput}
\nu(B(y,\varrho))\leq\Phi\big(c_0\:\cpct_{\alpha,p}(B(y,\varrho))\big)
\end{equation}
for all $y\in\mathbb{R}^n$ and $0<\varrho\leq r_0$. Then there is a constant $r_1>0$ such that for any Borel set $E\subset \mathbb{R}^n$ with $\diam(E)\leq r_1$ we have 
\begin{equation}\label{E:mazyaoutput}
\nu(E)\leq C_n\,\Phi\big(c_1\,\Theta(c(\Phi))^{p-1}\,\cpct_{\alpha,p}(E)\big),
\end{equation}
where $C_n>0$ is a constant depending only on $n$, $c_1>0$ is constant depending only on $n$, $p$, $\alpha$ and $c_0$, $\Theta:(0,\infty)\to (0,\infty)$ is defined by
\[\Theta(s)=\begin{cases} \frac{s+p-1}{n-\alpha p}\ &\text{if $\alpha p<n$},\\ s+1+\frac{1}{e}\ &\text{if $\alpha p=n$},\end{cases}
\]
and $c(\Phi)$ is as in \eqref{E:intcondition}.
\end{lemma}
}

The inequality \eqref{E:mazyaoutput} encodes an embedding, see \cite[Theorem 7.2.1]{AH96} or \cite[Section 8.6]{Mazya85}, a classical variant was used in \cite{Caetano2000, JW84, Wallin91, Ziemer}.

{ 
\begin{proof}
The original \cite[Section 8.5, Corollary 2]{Mazya85} follows from \cite[Section 8.5, Theorem]{Mazya85} and \cite[Section 8.5, Corollary 1]{Mazya85}.
The assumptions on $\Phi$ stated before Examples \ref{Ex:keyex} ensure that the conclusions of \cite[Section 8.5, Theorem]{Mazya85} and its corollaries hold; the proof there uses only the monotonicity of $\Phi$ and inequality \eqref{E:intcondition}. Since the behaviour of constants is not tracked in \cite{Mazya85}, but needed for Lemma \ref{L:Mazya}, we revisit the proofs in \cite{Mazya85} and provide more detailed comments. 

By the monotonicity of $\Phi$ and the regularity properties of $\nu$ and $\cpct_{\alpha,p}$ it suffices to prove \eqref{E:mazyaoutput} for compact sets $E$. We therefore assume that $E$ is compact.

Suppose first that $\alpha p<n$. We write $\|u\|_{\dot{H}^\alpha_p(\mathbb{R}^n)}:=\|(|\cdot|^\alpha \hat u)^\vee\|_{L^p(\mathbb{R}^n)}$, $u\in C_c^\infty(\mathbb{R}^n)$. Given a compact set $K\subset \mathbb{R}^n$, we consider the capacity
\[\dcpct_{\alpha,p}(K):=\inf\big\lbrace  \|u\|_{\dot{H}^\alpha_p(\mathbb{R}^n)}^p:\ u\in C_c^\infty(\mathbb{R}^n),\ u\geq 1\ \text{on $K$}\big\rbrace,\]
and we extend $\dcpct_{\alpha,p}$ to general subsets of $\mathbb{R}^n$ in the same way as $\cpct_{\alpha,p}$ was extended in Section \ref{S:results}.

Choose $r_1:=r_0$ and let $\gamma>0$ be such that $\cpct_{\alpha,p}(B(0,\varrho))\leq \gamma \dcpct_{\alpha,p}(B(0,\varrho))$ for all $\varrho\leq r_1$. It is well known that such a $\gamma$ exists \cite[Proposition 5.1.4]{AH96}. Since $\diam(E)\leq r_1$, the inequality \eqref{E:mazyainput} and the monotonicity of both $\Phi$ and the capacities imply that the restriction $\nu_1:=\nu|_E$ of $\nu$ to $E$ satisfies 
\begin{equation}\label{E:nuone}
\nu_1(B(y,\varrho))\leq\Phi\big(c_0\gamma\:\dcpct_{\alpha,p}(B(y,\varrho))\big)\quad \text{for all $y\in\mathbb{R}^n$ and $\varrho>0$},
\end{equation}
a detailed proof can be found in \cite[Section 8.5, Remark 2]{Mazya85}. Estimate \eqref{E:nuone} and the arguments in the proofs of \cite[Section 8.5, Theorem and Corollary 1]{Mazya85} (which we discuss in more detail below) show that the compact set $E$ admits a finite cover by open balls $B(x_k,r_k)$ such that $\nu(E)=\nu_1(E)$ is bounded by
\begin{equation}\label{E:intermediate}
 \sum_k \Phi(c_0\gamma\dcpct_{\alpha,p}(B(x_k,r_k))\leq C_n\Phi(c_1\Theta(c(\Phi))^{p-1}\dcpct_{\alpha,p}(E))
\end{equation}
where $c_1:=2^{p-1}c_0\omega\gamma$ with $\omega:=\dcpct_{\alpha,p}(B(0,1))$. Note that $\dcpct_{\alpha,p}(B(0,r))=\omega r^{n-\alpha p}$, $r>0$
\cite[Proposition 5.1.2]{AH96}. Inequality \eqref{E:intermediate}, together with the fact that $\dcpct_{\alpha,p}\leq \cpct_{\alpha,p}$  \cite[Proposition 5.1.4]{AH96}, and the monotonicity of $\Phi$, yields \eqref{E:mazyaoutput}. 

Given a Radon measure $\nu'$ on $\mathbb{R}^n$, let 
\[\dot{W}_{\alpha,p}\nu'(y):=\int_0^\infty\Big(\frac{\nu'(B(y,\varrho))}{\varrho^{n-\alpha p}}\Big)^{p'-1}\frac{d\varrho}{\varrho},\quad y\in\mathbb{R}^n.\]
Setting 
\begin{equation}\label{E:Calphap}
\dot{C}_{\alpha,p}(K):=\inf\Big\lbrace \int_{\mathbb{R}^n}\dot{W}_{\alpha,p}\nu'd\nu':\ \text{$\nu'$ is a Radon measure and $\dot{W}_{\alpha,p}\nu'\geq 1$ $(\alpha,p)$-q.e. on $K$}\Big\rbrace
\end{equation}
for a given compact set $K\subset \mathbb{R}^n$, we obtain a capacity $\dot{C}_{\alpha,p}$ equivalent to $\dcpct_{\alpha,p}$. This is a consequence of \cite[Theorem 4.5.4]{AH96}. The proof of \eqref{E:intermediate} in \cite[Corollary 1]{Mazya85} is based on an application of \cite[Section 8.5, Theorem]{Mazya85} to the equilibrium measure $\nu_E$ of $E$ with respect to $\dot{C}_{\alpha,p}$ and to a number $m\in (\frac12,1)$. To recall key arguments of the proof of \cite[Section 8.5, Theorem]{Mazya85}, we denote the total mass $\nu_E(\mathbb{R}^n)=\dot{C}_{\alpha,p}(E)$ of $\nu_E$ by $Q$ and let $\varrho_0>0$ be determined by $\varrho_0^{n-\alpha p}=\Theta(c(\Phi))^{p-1}m^{1-p}Q$. Substituting $t=(c_0 \omega\gamma)^{-1}\varrho^{\alpha p-n}$ in \eqref{E:intcondition} with $u=(c_0\omega\gamma)^{-1}\varrho_0^{\alpha p-n}$ gives 
\[\Phi(c_0\omega \gamma\varrho_0^{n-\alpha p})^{1-p'}\int_0^{\varrho_0}\Big(\frac{\Phi(c_0\omega \gamma\varrho^{n-\alpha p})}{\varrho^{n-\alpha p}}\Big)^{p'-1}\frac{d\varrho}{\varrho}\leq \frac{1}{n-\alpha p}c(\Phi)\varrho_0^{\frac{n-\alpha p}{1-p}}.\]
Using this, we can see that
\begin{equation}\label{E:dombym}
Y[\varphi]:=\int_0^\infty \Big(\frac{\varphi(\varrho)}{\varrho^{n-\alpha p}}\Big)^{p'-1}\frac{d\varrho}{\varrho}\leq \Theta(c(\Phi))Q^{p'-1}\varrho_0^{\frac{n-\alpha p}{1-p}}=m,
\end{equation}
where the function $\varphi$ is defined by $\varphi(\varrho):=Q\Phi(c_0\omega\gamma \varrho^{n-\alpha p})/\Phi(c_0\omega\gamma \varrho_0^{n-\alpha p})$ for $0<\varrho\leq \varrho_0$ and $\varphi(\varrho):=Q$ for $\varrho>\varrho_0$. Obviously the set $E$ is contained in the union $G\cup E_0$ of the sets 
\[G:=\{x\in \mathbb{R}^n:\ \dot{W}_{\alpha,p}\nu_E(x)\geq m\}\quad \text{and}\quad E_0:=\{x\in E:\ \dot{W}_{\alpha,p}\nu_E(x)< m\}.\]
By \cite[Section 8.5, Lemma]{Mazya85} the set $\{x\in\mathbb{R}^n:\ \dot{W}_{\alpha,p}\nu_E(x)\geq Y[\varphi]\}$ admits a cover by open balls $B(x_k,r_k)$ such that $\sum_k\varphi(r_k)\leq \frac12 C_n\nu_E(\mathbb{R}^n)$ holds with a constant $C_n>0$ depending only on $n$, and by the inequality in \eqref{E:dombym} this is also a cover for $G$. Since $\nu_E$ is the minimizer for the right-hand side of \eqref{E:Calphap} and, in particular, satisfies the conditions on $\nu'$ there, we have
$\dot{W}_{\alpha,p}\nu_E\geq 1$ $(\alpha,p)$-q.e. on $E$, which implies that $\dot{C}_{\alpha,p}(E_0)=0$, and by the equivalence of the capacities also $\dcpct_{\alpha,p}(E_0)=0$. Using the finiteness of the integral in \eqref{E:intcondition} for $u=1$ and \cite[Proposition 7.2.3/2]{Mazya85}, one can now conclude that for any $\varepsilon>0$ the set $E_0$ admits a cover by open balls $B(y_i,\varrho_i)$ so that $\sum_i \Phi(c_0\gamma\dcpct_{\alpha,p}(B(y_i,\varrho_i))<\varepsilon$ \cite[Section 8.5, Corollary 1]{Mazya85}. Combining these two covers gives a cover for $E$, and with $\varepsilon$ chosen sufficiently small the inequality \eqref{E:intermediate} follows.

Suppose next that $\alpha p=n$. Let $\omega>1$ be such that $\omega^{-1}|\log \varrho|\leq \cpct_{\alpha,p}(B(0,\varrho))\leq \omega|\log \varrho|$ for all $0<\varrho\leq \frac12$, cf. \cite[Proposition 5.1.3]{AH96}. Let $C_{\alpha,p}$ be defined as in \eqref{E:Calphap}, but with 
\[W_{\alpha,p}\nu'(y):=\int_0^\infty\Big(\frac{\nu'(B(y,\varrho))}{\varrho^{n-\alpha p}}\Big)^{p'-1}e^{-\varrho}\frac{d\varrho}{\varrho},\quad y\in\mathbb{R}^n.\]
in place of $\dot{W}_{\alpha,p}$. Then $C_{\alpha,p}$ is a capacity equivalent to $\cpct_{\alpha,p}$ \cite[Theorem 4.5.2]{AH96}. As before, let $m\in (\frac12,1)$. Choose $0<r_1<\min (\frac12,r_0)$ small enough to ensure that $C_{\alpha,p}(E)\leq c(\Phi)^{1-p}m^{p-1}$. Let $Q$ be the total mass $\nu_E(\mathbb{R}^n)=C_{\alpha,p}(E)$ of the equilibrium measure $\nu_E$ of $E$, now with respect to $C_{\alpha,p}$. Choose $0<\varrho_0<\frac{1}{e}$ such that $|\log \varrho_0|^{1-p}=\Theta(c(\Phi))^{p-1}m^{1-p}Q$. Substituting 
$t=(c_0\omega)^{1-p'}|\log \varrho|$ in \eqref{E:intcondition} with $u=(c_0\omega)^{1-p'}|\log \varrho_0|>(c_0\omega)^{1-p'}=:u_0$, we obtain 
\[(\Phi(c_0\omega|\log \varrho_0|^{1-p}))^{1-p'}\int_0^{\varrho_0}(\Phi(c_0\omega|\log \varrho|^{1-p}))^{1-p'}\frac{d\varrho}{\varrho}\leq c(\Phi)|\log\varrho_0|.\]
This now gives 
\[\int_0^\infty(\varphi(\varrho))^{p'-1}e^{-r}\frac{dr}{r}\leq \Theta(c(\Phi))Q^{p'-1}|\log \varrho_0| \leq m,\]
where $\varphi(\varrho):=Q\Phi(c_0\omega|\log\varrho|^{1-p})/\Phi(c_0\omega|\log\varrho_0|^{1-p})$ for $0<\varrho<\varrho_0$ and $\varphi(\varrho):=Q$ for $\varrho>\varrho_0$. Applying \cite[Section 8.5, Lemma and Theorem]{Mazya85} similarly as before and again proceeding as in the proof of \cite[Section 8.5, Corollary 1]{Mazya85}, we see that $E$ admits a ball cover for which an analog of \eqref{E:intermediate} holds. This gives \eqref{E:mazyaoutput} with $c_1=2^{p-1}c_0\omega$.
\end{proof}
}

%
%
%

From Lemma \ref{L:Mazya} we deduce suitable localized estimates.

\begin{corollary}\label{C:isoper}
{Let $1<p<\infty$, let $\mu$ be a  non-zero Radon measure on $\mathbb{R}^n$, $\Gamma=\supp\mu$ and $d:\Gamma\to [0,\infty)$ a Borel function.} 
\begin{enumerate}
\item[(i)] If $0<\alpha<\frac{n}{p}$ and {\eqref{E:fatDbar} and \eqref{E:traceconddx}} hold, then for any $x\in \Gamma$ there are constants $c_2(x)>0$ and $r_2(x)>0$ such that for any Borel set $E\subset \mathbb{R}^n$ and any {$0<\varrho\leq r\leq r_2(x)$} we have 
\begin{equation}\label{E:isoperd}
\mu(B(x,\varrho)\cap E)\leq c_2(x)\cpct_{\alpha,p}(B(x,\varrho)\cap E)^{{ \frac{\underline{q}(x,r)}{p}}},
\end{equation}
where $\underline{q}(x,r):=\frac{p\underline{d}(x,r)}{n-\alpha p}$. If $\alpha=\frac{n}{p}$ and {\eqref{E:fatDbar} and \eqref{E:traceconddx}} hold, then for any $x\in \Gamma$ there are constant $c_2(x)>0$, $r_2(x)>0$ such that for any Borel set $E\subset \mathbb{R}^n$ and any  $0<\varrho\leq r\leq r_2(x)$ we have
\begin{equation}\label{E:isoperexp}
\mu(B(x,\varrho)\cap E)\leq c_2(x)\exp\left(-{c_3}\cpct_{\alpha,p}(B(x,\varrho)\cap E)^{1-p'}\right),
\end{equation}
where {$\underline{d}=\inf_{x\in\Gamma}d(x)$ and $c_3>0$ is a constant independent of $x$, $r$ and $\varrho$.}
\item[(ii)] If $\alpha=\frac{n}{p}$ and  {\eqref{E:fatDbarkappa} and \eqref{E:tracecondhx}} hold, then the conclusion in (i) holds with {$\underline{q}(x,r):=\frac{p\underline{\kappa}(x,r)}{p-1}$ in \eqref{E:isoperd}}.
\end{enumerate}
\end{corollary}

\begin{proof}{ 
To see (i), let $x\in \Gamma$ and $0<r\leq r_2(x):=\frac12\min(r(x),1)$, where $r(x)$ is as in \eqref{E:upregdx}. Consider the finite Borel measure $\nu:=\mu|_{B(x,r)}$. Similarly as in Remark \ref{R:main} (i) we can see that \eqref{E:upregdx} implies 
\begin{equation}\label{E:aboveestimate}
\nu(B(y,\varrho))\leq 2^{d(x)}c(x)\varrho^{\underline{d}(x,r)},\quad y\in \mathbb{R}^n,\ 0<\varrho\leq r. 
\end{equation}

Suppose first that $0<\alpha<\frac{n}{p}$. Then $\cpct_{\alpha,p}(B(y,\varrho))\geq \omega\,\varrho^{n-\alpha p}$ for all $\varrho>0$ and $y\in\mathbb{R}^n$, where $\omega>0$ is as defined after \eqref{E:intermediate}. Together with \eqref{E:aboveestimate} this gives 
\[\nu(B(y,\varrho))\leq 2^{d(x)}c(x)(\omega^{-1}\cpct_{\alpha,p}(B(y,\varrho)))^{\underline{q}(x,r)/p},\quad y\in \mathbb{R}^n,\ 0<\varrho\leq r,\]
which is \eqref{E:mazyainput} with $r_0=r$, $c_0=\omega^{-1}$, {and 
\[\Phi_{x,r}(t):=2^{d(x)}c(x)t^{\underline{q}(x,r)/p}.\] 
Condition \eqref{E:traceconddx} ensures that $p<\underline{q}(x,r)$. Writing $q(x):=\frac{pd(x)}{n-\alpha p}$ and $q_-(x):=q(x,r_2(x))$, we have $q_-(x)\leq \underline{q}(x,r)\leq q(x)$ for all $0<r\leq r_2(x)$. 
The function $\Phi_{x,r}$ is of the type discussed in  Example \ref{Ex:keyex} (i), and reasoning as there, we find that 
$c(\Phi_{x,r})\leq c(\Phi_{x,r_2(x)})=(p-1)/(\frac{q_-(x)}{p}-1)$ 
for all $0<r\leq r_2(x)$. We assume, without loss of generality, that the constant $c_1$ in  \eqref{E:mazyaoutput} is larger than $\Theta(c(\Phi_{x,r_2(x)}))^{1-p}$. Then Lemma \ref{L:Mazya} gives \eqref{E:isoperd} with the constant 
\[c_2(x)=2^{d(x)}c(x)c(n)\big(c_1\Theta(c(\Phi_{x,r_2(x)}))^{p-1}\big)^{\frac{q(x)}{p}},\]
which does not depend on $r$.

Now suppose that $\alpha=\frac{n}{p}$. In this case 
\begin{equation}\label{E:logcap}
\cpct_{\alpha,p}(B(y,\varrho))\geq \omega^{-1}\,(-\log \varrho)^{1-p} 
\end{equation}
for all $0<\varrho\leq r$ and $y\in\mathbb{R}^n$ with $\omega$ as in the proof of Lemma \ref{L:Mazya} for the case $\alpha p=n$. This fact, together with \eqref{E:aboveestimate}, gives
\[\nu(B(y,\varrho))\leq 2^{d(x)}c(x)\exp(-\underline{d}(\omega \cpct_{\alpha,p}(B(y,\varrho)))^{1-p'}),\quad y\in \mathbb{R}^n,\ 0<\varrho\leq r,\]
which is \eqref{E:mazyainput} with $r_0=r$, $c_0=\omega$ and 
\begin{equation}\label{E:Phi}
\Phi_x(t)=2^{d(x)}c(x)\exp(-\underline{d}t^{1-p'}). 
\end{equation}
By \eqref{E:traceconddx}, Example \ref{Ex:keyex} (ii) and Lemma \ref{L:Mazya} and its proof we obtain \eqref{E:isoperexp} with \[c_2(x)=2^{d(x)}c(x)c(n)\quad \text{and}\quad c_3=\underline{d}c_1^{1-p'}/\big(\underline{d}(p'-1)(c_0\omega)^{p-1}+1+\frac{1}{e}\big).\]

To see (ii) we can proceed similarly. Recall estimate \eqref{E:upreghx}. For $\nu$ as defined above, we find that 
\[\nu(B(y,\varrho))\leq c(x)2^{\kappa(x)}(-\log \varrho)^{-\underline{\kappa}(x,r)}, \quad y\in \mathbb{R}^n,\ 0<\varrho\leq r\leq r_2(x),\]
where $r_2(x):=\frac14\min\{r(x),1\}$. Combining with \eqref{E:logcap}, 
\[\nu(B(y,\varrho))\leq c(x)2^{\kappa(x)}(\omega\cpct_{\alpha,p}(B(y,\varrho)))^{\frac{\underline{\kappa}(x,r)}{p-1}},\quad y\in \mathbb{R}^n,\ 0<\varrho\leq r,\]
which is again \eqref{E:mazyainput}, now with $r_0=r$, $c_0=\omega$ and $\Phi_{x,r}(t)=2^{\kappa(x)}c(x)t^{\underline{q}(x,r)/p}$. Again condition \eqref{E:traceconddx} and Lemma \ref{L:Mazya} give \eqref{E:isoperd}, now with 
\[c_2(x)=2^{\kappa(x)}c(x)c(n)\big(c_1\Theta(c(\Phi_{x,r_2(x)}))^{p-1}\big)^{\frac{\underline{q}(x)}{p}},\]
where $\underline{q}(x):=\frac{p\underline{\kappa}(x)}{p-1}$.
}}
%
\end{proof}

\begin{remark} Recall the notion of $(\alpha,p)$-equivalence from Remark \ref{R:counter}. Under the hypotheses of Lemma \ref{L:smooth} any set $E\subset \mathbb{R}^n$ which is  $(\alpha,p)$-equivalent to a Borel set $A\subset \mathbb{R}^n$ is an element of the completion of the Borel $\sigma$-field on $\mathbb{R}^n$ with respect to $\mu$ and we have $\mu(E\Delta A)=0$. By the subadditivity of $\cpct_{\alpha,p}$ and the monotonicity of the functions $\Phi$, the conclusions of Corollary \ref{C:isoper} remain true if we only require that $E\subset \mathbb{R}^n$ is $(\alpha,p)$-equivalent to a Borel set.
\end{remark}

\section{Thin sets, fine continuity and averages}\label{S:proof}

In this section we provide our proofs of Theorems \ref{T:main} and \ref{T:main_logcase}. Our arguments use average limits based on fine continuity (defined below). The existence of these average limits is ensured by assumptions \eqref{E:suffconddx} respectively \eqref{E:suffcondh} and the following lemma. 
Recall that, given $1<p<\infty$ and $0<\alpha\leq \frac{n}{p}$,
a set $E\subset \mathbb{R}^n$ is said to be \emph{$(\alpha,p)$-thin at $x\in\mathbb{R}^n$} if 
\begin{equation}\label{E:thin}
\int_0^1\left(\frac{\cpct_{\alpha,p}(B(x,r)\cap E)}{r^{n-\alpha p}}\right)^{p'-1}\frac{dr}{r}<\infty,
\end{equation}
see \cite[Definition 6.3.7]{AH96}. Original sources for the $(\alpha,p)$-case are \cite{AdamsMeyers, HedbergWolff, Meyers75}. Classical results on thin sets and further context can be found in \cite{Brelot40, Brelot44, PSz45, Wiener} and in particular in \cite[Chapitre IV]{Deny50}.

\begin{lemma}\label{L:lessthanone}
{Let $1<p<\infty$, let $\mu$ be a  non-zero Radon measure on $\mathbb{R}^n$, $\Gamma=\supp\mu$ and $d:\Gamma\to [0,\infty)$ a Borel function.}
Let 
$E\subset \mathbb{R}^n$ be a Borel set that is $(\alpha,p)$-thin at a point $x\in \Gamma$. If
\begin{enumerate}
\item[(i)] $0<\alpha\leq \frac{n}{p}$, {\eqref{E:L}, \eqref{E:fatDbar} and \eqref{E:traceconddx}} hold and $\bar{D}^{d(x)}\mu(x)>0$ or
\item[(ii)] $\alpha=\frac{n}{p}$, {\eqref{E:Lkappa}, \eqref{E:fatDbarkappa} and \eqref{E:tracecondhx}} hold and $\bar{D}^{h_{\kappa(x)}}\mu(x)>0$,
\end{enumerate}
then there is a sequence $(r_k)$ such that $r_k\downarrow 0$ as $k\to \infty$ and  
\begin{equation}\label{E:lessthanone}
\sup_k\frac{\mu(B(x,r_k)\cap E)}{\mu(B(x,r_k))}<1.
\end{equation}
{In particular, we have $\mu(B(x,r)\setminus E)>0$ for all $r>0$.}
\end{lemma}

\begin{proof}
Suppose first that $0<\alpha<\frac{n}{p}$. Substituting $t=r^{\frac{\alpha p-n}{p-1}}$, we see that 
\eqref{E:thin} is equivalent to the integrability of the decreasing function $\varphi(t):=(\cpct_{\alpha,p}(B(x,t^{\frac{p-1}{\alpha p-n}})\cap E)^{p'-1}$ over $(1,\infty)$, which forces $\lim_{t\to\infty}t\varphi(t)=0$, cf. \cite[p. 170]{Deny50}. {Together with \eqref{E:isoperd}, this gives 
\[\lim_{r\to 0}\frac{\mu(B(x,r)\cap E)^{p/\underline{q}(x,r)}}{r^{n-\alpha p}}\leq c_4(x)\:\lim_{r\to 0}\frac{\cpct(B(x,r)\cap E)}{r^{n-\alpha p}}=0\]
with a suitable constant $c_4(x)>0$. Using \eqref{E:L} and arguing as in Remark \ref{R:basicimplication} (i), we arrive at
\[\limsup_{r\to 0} \frac{\mu(B(x,r)\cap E)}{r^{d(x)}}\leq e^{L_d(x)}\lim_{r\to 0}\frac{\mu(B(x,r)\cap E)}{r^{\underline{d}(x,r)}}=0.\]
Since the upper density is positive, there is a sequence $(r_k)$ with $r_k\downarrow 0$ as $k\to \infty$ and some $\varepsilon>0$ such that $\mu(B(x,r_k))>\varepsilon\: r_k^{d(x)}$ for all $k$. Combining, we obtain \eqref{E:lessthanone}. 

Now suppose that $\alpha=\frac{n}{p}$. By \eqref{E:thin} and the substitution $t=-\log r$ we have 
\begin{equation}\label{E:Wienerlog}
\lim_{r\to 0}\cpct_{\alpha,p}(B(x,r)\cap E)(-\log r)^{p-1}=0.
\end{equation}
Clearly $\lim_{s\to 0}\Phi(s)=0$ holds for the function $\Phi$ in \eqref{E:Phi}, so that \eqref{E:isoperexp} and \eqref{E:Wienerlog} give
\[\limsup_{r\to 0}\Big(\frac{1}{c_2(x)}\mu(B(x,r)\cap E)\Big)^{\frac{1}{-\log r}}\leq \lim_{r\to 0}\exp\Big(-c_3\cpct_{\alpha,p}(B(x,r)\cap E)^{1-p'}(-\log r)^{-1}\Big)=0\]
Taking the logarithm and observing that $c_2(x)$ may be omitted, we find that 
\[\lim_{r\to 0}(-\log r)\big(-\log \mu(B(x,r)\cap E)\big)^{-1}=0.\]
The positivity of the upper density gives
\[(-2 d(x)\log r_k)^{-1}\leq \big(-\log\mu(B(x,r_k))\big)^{-1}\]
along a sequence $(r_k)$ as before. Combining, we see that 
\[\big(-\log\mu(B(x,r_k)\cap E)\big)^{-1}< (-\log \mu(B(x,r_k))^{-1}\]
for all sufficiently large $k$. Taking exponentials on both sides gives \eqref{E:lessthanone}. 

To prove (ii) one can proceed similarly, combining \eqref{E:isoperd} with {$\underline{q}(x,r)$ as in Corollary \ref{C:isoper} (ii) and \eqref{E:Wienerlog} to arrive at 
\[\limsup_{r\to 0}\mu(B(x,r)\cap E)(-\log r)^{\kappa(x)}\leq e^{\ell_\kappa(x)}\lim_{r\to 0}\mu(B(x,r)\cap E)(-\log r)^{\underline{\kappa}(x,r)}=0.\]
By the positivity of the upper density we can again find a sequence  $(r_k)$ as before and some $\varepsilon>0$ such that 
$\mu(B(x,r_k))>\varepsilon(-\log r_k)^{-\kappa(x)}$ for all $k$. Together this gives \eqref{E:lessthanone}. }
}
\end{proof}

\begin{remark}
Lemma \ref{L:lessthanone} actually holds for $E\subset \mathbb{R}^n$ $(\alpha,p)$-equivalent to a Borel set that is $(\alpha,p)$-thin at $x$.
\end{remark}

Recall that a function $g:\mathbb{R}^n\to\mathbb{R}$ is said to be \emph{$(\alpha,p)$-finely continuous at $x\in\mathbb{R}^n$} if for any $\varepsilon>0$ the set $E_x^\varepsilon(g):=\{y\in\mathbb{R}^n:\ |g(x)-g(y)|\geq \varepsilon\}$ is $(\alpha,p)$-thin at $x$. See \cite[Definition 6.4.2]{AH96}, plus \cite{Hedberg72, HedbergWolff, Meyers75} and \cite{Brelot71, Fuglede65, Fuglede71} for further background. It is well known that if $g$ is $(\alpha,p)$-finely continuous at $x\in\mathbb{R}^n$, then there is a set $E_x\subset\mathbb{R}^n$, $(\alpha,p)$-thin at $x$ and such that $\lim_{y\to x,\ y\in  \mathbb{R}^n\setminus E_x}g(y)=g(x)$, 
see \cite[Proposition 6.4.3]{AH96}. The proof of this result shows that one can take $E_x$ to be the union of intersections $E_x^{\varepsilon_k}(g)\cap B(x,r_k)$, where $(\varepsilon_k)$ and $(r_k)$ are sequences with $\varepsilon_k\downarrow 0$ and $r_k\downarrow 0$ as $k\to \infty$, see  \cite[p. 177]{AH96}. {In particular, we can take $E_x$ to be Borel, provided that $g$ is Borel.} It is also {well known} that an $(\alpha,p)$-quasi continuous function is $(\alpha,p)$-finely continuous at $(\alpha,p)$-q.e.\ $x\in\mathbb{R}^n$, see \cite[Theorem 6.4.5]{AH96}, \cite{Hedberg72, HedbergWolff, Meyers75} and \cite{Deny50, Fuglede65}.

We can now provide our promised proofs of Theorems \ref{T:main} and \ref{T:main_logcase}. 

\begin{proof}[Proof of Theorems \ref{T:main} and \ref{T:main_logcase}]
By Lemma \ref{L:smooth} it suffices to prove that the vanishing of $\widetilde{u}$ {$\mu$-a.e.} on $\Gamma$ implies its vanishing $(\alpha,p)$-q.e.\ on $\Gamma$.

{As mentioned in Section \ref{S:results}, we can find an $(\alpha,p)$-quasi continuous representative $\widetilde{u}$ which is a {real-valued} Borel function on $\mathbb{R}^n$.} This follows from \cite[Propositions 2.3.7 and 6.1.3]{AH96}. By the preceding we can find a set $N_0\subset \mathbb{R}^n$ of zero $(\alpha,p)$-capacity such that $\widetilde{u}$ is $(\alpha,p)$-finely continuous at all $x\in \mathbb{R}^n\setminus N_0$. For any such $x$, there is a Borel set $E_x\subset\mathbb{R}^n$, $(\alpha,p)$-thin at $x$, such that 
\begin{equation}\label{E:finelim}
\lim_{y\to x,\ y\in  \mathbb{R}^n\setminus E_x}\widetilde{u}(y)=\widetilde{u}(x).
\end{equation}

By \eqref{E:suffconddx} respectively \eqref{E:suffcondh} we can find a set $N_1\subset \mathbb{R}^n$ with $\cpct_{\alpha,p}(N_1)=0$ such that at all $x\in \Gamma\setminus N_1$ we have  $\bar{D}^{d(x)}\mu(x)>0$ respectively $\bar{D}^{h_{\kappa(x)}}\mu(x)>0$.

The set $N:=N_0\cup N_1$ has zero $(\alpha,p)$-capacity. Now let $x\in \Gamma\setminus N$. {Using Lemma \ref{L:lessthanone} and \eqref{E:finelim} we obtain 
\[\widetilde{u}(x)=\lim_{r\to 0} \frac{1}{\mu(B(x,r)\setminus E_x)}\int_{B(x,r)\setminus E_x} \widetilde{u}(y)\:d\mu(y).\]
}
If now $\widetilde{u}=0$ $\mu$-a.e.\ on $\Gamma=\supp\mu$, then this gives $\widetilde{u}(x)=0$. It follows that $\widetilde{u}=0$ on $\Gamma\setminus N$, and this confirms the desired implication.
\end{proof}


\begin{thebibliography}
\normalsize
\bibitem{AH96}
D.~R. Adams, L.~I. Hedberg, \emph{Function {S}paces and {P}otential {T}heory},
  Grundlehren math. Wiss. vol. 314, Springer, Berlin, 1996.
\bibitem{AdamsMeyers}
D.~R. Adams, N.~G. Meyers, \emph{Thinness and Wiener criteria for non-linear potentials},
Indiana Univ. Math. J. {\bf 22} (2) (1972), 169--197.
\bibitem{ArendtWarma03}
W. Arendt, M. Warma, \emph{The Laplacian with Robin boundary conditions on
arbitrary domains}, Pot. Anal. {\bf 19} (4) (2003), 341--363.
\bibitem{Bagby72}
T. Bagby, \emph{Quasi topologies and rational approximation}, J. Funct. Anal. {\bf 10} (1972)
259--268.
\bibitem{Beurling49}
A. Beurling, \emph{On the spectral synthesis of bounded functions},
Acta Math. {\bf 89} (1949), 225--238.
\bibitem{Biegert09}
M. Biegert, \emph{On traces of Sobolev functions on the boundary of extension domains}, Proc. Amer. Math. Soc. {\bf 137} (12) (2009), 4169--4176.
\bibitem{Biegert09b}
M. Biegert, \emph{On a capacity for modular spaces}, J. Math. Anal. Appl. {\bf 358} (2) (2009), 294--306.
\bibitem{Brelot40}
M. Brelot, \emph{Points irr\'eguliers et transformations continues en th\'eorie du potentiel}, Journ.
Math. Pures Appl. {\bf 19} (1940), 319--337.
\bibitem{Brelot44}
M. Brelot, \emph{Sur les ensembles effil\'es}, Bull. Sci. Math. {\bf 68} (1944), 12--36.
\bibitem{Brelot71}
M. Brelot, \emph{On topologies and boundaries in potential theory}, Lect. Note Math. {\bf 175}, Springer, Berlin, 1971.
\bibitem{Caetano2000}
A.~M. Caetano, \emph{Approximation by functions of compact support in Besov-Triebel-Lizorkin spaces
on irregular domains}, Stud. Math. {\bf 142} (2000), 47--63.
\bibitem{Caetano2002}
A.~M. Caetano, \emph{On fractals which are not so terrible}, Fundam. Math. {\bf 171} (2002), 249--266.
\bibitem{CCh-WCGHM25}
A.~M. Caetano, S.~N. Chandler-Wilde, X. Claeys, A. Gibbs, D.~P. Hewett,  A. Moiola, \emph{Integral
equation methods for acoustic scattering by fractals}, Proc. R. Soc. A, {\bf 481} (2025), 20230650.
\bibitem{CCh-WGHM24}
A.~M. Caetano, S.~N. Chandler-Wilde, A. Gibbs, D.~P. Hewett, A. Moiola, \emph{A Hausdorff-measure
boundary element method for acoustics scattering by fractal screens} Numer. Math. {\bf 156} (2024), 463--
532.
\bibitem{CCh-WH25}
A.~M. Caetano, S.~N. Chandler-Wilde, D.~P. Hewett, \emph{Properties of IFS attractors with non-empty interiors and associated function spaces and scattering problems}, preprint, arXiv:2511.15213.
\bibitem{CHM21}
A.~M. Caetano, D.~P. Hewett, A. Moiola, \emph{Density results for Sobolev, Besov and
Triebel–Lizorkin spaces on rough sets}, J. Funct. Anal. {\bf 281} (2021), 109019.
\bibitem{Calderon61}
A. Calder\'on, \emph{Lebesgue spaces of differentiable functions and distributions}, In: Proc.
Symp. Pure Math., Vol. IV, ed.: Ch.~B. Morrey, Jr., Amer. Math. Soc., Providence, 1961, pp. 33--49,
\bibitem{Ch-WHM2017}
S.~N. Chandler-Wilde, D.~P. Hewett, A. Moiola, \emph{Sobolev spaces on non-Lipschitz subsets of $\mathbb{R}^n$
with application to boundary integral equations on fractal screens}, Integr. Equ. Oper. Theory,
{\bf} (2017), 179--224.
\bibitem{Ch-WHMB2021}
S.~N. Chandler-Wilde, D.~P. Hewett, A. Moiola, J. Besson, \emph{Boundary element methods for acoustic scattering by
fractal screens}, Numer. Math. {\bf 147} (2021), 785-837.
\bibitem{Ch-WCHR-PS2025}
S.~N. Chandler-Wilde, G. Claret, D.~P. Hewett, A. Rozanova-Pierrat, S. Sadeghi, {\emph{Integral equation methods for scattering by multifractal obstacles}, preprint (2026), arXiv:2605.19540.}
\bibitem{CHR-PT24}
G. Claret, M. Hinz, A. Rozanova-Pierrat, A. Teplyaev, \emph{Layer potential operators for transmission problems on extension domains}, Journ. Math. Pures Appl. (2026), 103888.
\bibitem{Deny50}
J. Deny, \emph{Les potentiels d'\'energie finie}, Acta Math. {\bf 82} (1950), 107--183.   
\bibitem{Falconer97}
K.~J. Falconer, \emph{Techniques in Fractal Geometry}, Wiley, Chichester, 1997. 
\bibitem{FarkasJacob2001}
W. Farkas, N. Jacob, \emph{Sobolev spaces on non-smooth domains and Dirichlet forms
related to subordinate reﬂecting diffusions}, Math. Nachr. {\bf 224} (2001), 75 -- 104.
\bibitem{Fuglede65}
B. Fuglede, \emph{Quasi topology and fine topology}, In: Seminaire Brelot-Choquet-Deny. Th\'eorie du Potentiel,
{\bf 10} (2) (1965-1966) no. 12, 1--14.
\bibitem{Fuglede71}
B. Fuglede, \emph{The quasi topology associated with a countably additive set
function}, Ann. Inst. Fourier, Grenoble, {\bf 21} (1) (1971), 123--169.
\bibitem{FukushimaKaneko85}
M. Fukushima, H. Kaneko, \emph{On $(r,p)$-capacities for general Markovian semigroups}, In: Inﬁnite Dimensional Analysis and Stochastic Processes, ed. by S. Albeverio, Res. Notes Math. 124, Pitman, Boston (MA), 1985, pp. 41--47.
\bibitem{FukushimaLeJan}
M. Fukushima, Y. LeJan, \emph{On quasi-supports of smooth measures and closability of pre-Dirichlet forms},
Osaka J. Math. {\bf 28} (1991), 837--845.
\bibitem{FOT2011}
M. Fukushima, Y. Oshima, M. Takeda, \emph{Dirichlet Forms and Symmetric Markov Processes}, 2nd edition, Studies in Math. 19, 
deGruyter, Berlin, New York, 2011.
\bibitem{Hajlaszetal08}
P. Haj\l asz, P. Koskela, H. Touminen, \emph{Sobolev embeddings, extensions and measure density condition},
J. Funct. Anal. {\bf 254} (2008) 1217--1234.
\bibitem{Harjulehto2006}
{P. Harjulehto, P. Hästö, V. Latvala, \emph{Sobolev embeddings in metric measure
spaces with variable dimension}, Math. Z. {\bf 254} (2006), 591--609.}
\bibitem{Havin68}
V.~P. Havin, \emph{Approximation in the mean by analytic functions}, Soviet Math. Dokl. {\bf 9}
(1968, 245--248.
\bibitem{Hedberg72}
L.~I. Hedberg, \emph{Non-linear potentials and approximation in the mean by analytic
functions}, Math. Z. {\bf 129} (1972), 299--319.
\bibitem{Hedberg81}
L.~I. Hedberg, \emph{Spectral synthesis in Sobolev spaces, and uniqueness of solutions of the Dirichlet problem},
Acta Math. {\bf 147} (1981), 237--264.
\bibitem{HedbergNetrusov07}
L.~I. Hedberg, Y. Netrusov, \emph{An axiomatic approach to function spaces, spectral synthesis, and Luzin approximation}, Mem. Am. Math. Soc. {\bf 188} (2007), 1--97.
\bibitem{HedbergWolff}
L.~I. Hedberg, Th.~H. Wolff, \emph{Thin sets in nonlinear potential theory}, Ann. Inst.
Fourier {\bf 33} (4) (1983), 161--187.
\bibitem{HIT16}
T. Heikkinen, L. Ihnatsyeva, H. Tuominen, \emph{Measure density and extension of Besov and Triebel-Lizorkin functions}, 
J. Fourier Anal. Appl. {\bf 22} (2016), 334--382.
\bibitem{HM2017}
D.~P. Hewett, A. Moiola, \emph{On the maximal Sobolev regularity of distributions supported by subsets of Euclidean space}, Anal. Appl. {\bf 15} (5) (2017), 731--770. 
\bibitem{HKM17}
M. Hinz, S. Kang, J. Masamune, \emph{Probabilistic characterizations of essential self-adjointness and removability of singularities},  Sci. Journal of Volgograd State Univ. Math. Physics and Comp. Sim. 2017. Vol. 20 (3) (2017), 148-162.
\bibitem{HMS23}
M. Hinz, J. Masamune, K. Suzuki, \emph{Removable sets and Lp-uniqueness on manifolds and metric measure spaces}, Nonlinear Anal. {\bf 234} (2023), 113296.
\bibitem{HR-PT2021}
M. Hinz, A. Rozanova-Pierrat, A. Teplyaev, \emph{Non-Lipschitz uniform domain shape optimization in linear acoustics}, SIAM J. Control Opt. {\bf 59} (2) (2021), 1007--1032. 
\bibitem{HR-PT2023}
M. Hinz, A. Rozanova-Pierrat, A. Teplyaev, \emph{Boundary value problems on non-Lipschitz uniform domains: Stability, compactness and the existence of optimal shapes}, Asympt. Anal. {\bf 134} (2023), 25--61.
\bibitem{Hutchinson81}
{J.E. Hutchinson, \emph{Fractals and self-similarity}, Indiana Univ. Math. J. {\bf 30} (1981), 713--747.}
\bibitem{Jones81}
P.W. Jones, \emph{Quasiconformal mappings and extendability of functions in Sobolev spaces}, Acta
Math. {\bf 147} (1981), 71--88.
\bibitem{J79}
A. Jonsson, \emph{The trace of potentials on general sets}, Ark. Mat. {\bf 17} (1979), 1--18.
\bibitem{JW84}
A. Jonsson, H. Wallin, \emph{Function Spaces on Subsets of $\mathbb{R}^n$}, Math. Reports, Harwood Academic Publishers, Chur,  1984. 
\bibitem{Jonsson94}
A. Jonsson, \emph{Besov spaces on closed subsets of $\mathbb{R}^n$}, Trans. Amer. Math. Soc. {\bf 341} (1) (1994), 355--370.
\bibitem{Kalyabin85}
G.~A. Kalyabin, \emph{Theorems on extensions, multipliers and diffeomorphisms for generalized Sobolev-Liouville classes in domains with Lipschitz boundary}, Trudy Mat. Inst. Steklov {\bf 172} (1985), 173--186.
\bibitem{Landkof}
N.~S. Landkof, \emph{Foundations of Modern Potential Theory}, Springer, New York, 1972.
\bibitem{Marschall87}
J. Marschall, \emph{The trace of Sobolev-Slobodeckij spaces on Lipschitz domains}, Manuscripta Math. {\bf 58} (1987), 47--65.
\bibitem{Mattila}
P. Mattila, \emph{Geometry of Sets and Measures in Euclidean spaces. Fractals and Rectifiability}, Cambridge Studies in
Adv. Math., vol. 44, Cambridge University Press, Cambridge, 1995.
\bibitem{MauldinUrbanski96}
{R. D. Mauldin, M. Urba\'nski,\emph{Dimensions and measures in infinite iterated function systems}, Proc. London Math. Soc. {\bf 73} (1996), 105--154. }
\bibitem{Mazya85}
V.~G. Mazya, \emph{Sobolev Spaces}, Springer, Berlin, 1985.
\bibitem{Meyers75}
N.~G. Meyers, \emph{Continuity properties of potentials}, Duke Math. J. {\bf 42} (1975), 157--166.
\bibitem{Moran46}
{P. A. P. Moran, \emph{Additive functions of intervals and Hausdorff measure}, Math. Proc. Camb. Phil. Soc. {\bf 42} (1946), 15--23.}
\bibitem{Necas67}
J. Ne\v{c}as, \emph{Les M\'ethodes Directes en Th\'eorie des \'Equations Elliptiques}, Masson, Paris, 1967. 
\bibitem{Netrusov93}
Y. Netrusov, \emph{Spectral synthesis in spaces of smooth functions}, Russian Acad.
Sci. Dokl. Math. {\bf 46} (1993), 115--117.
\bibitem{PSz45}
G. Polya, G. Szeg\"o, \emph{Inequalities for the capacity of a condenser}, Amer. Jour. of Math.
{\bf 67} (1945), 1--32.
\bibitem{Rogers06}
L. Rogers, \emph{Degree-independent Sobolev extension on locally
uniform domains}, J. Funct. Anal. {\bf  235} (2006), 619--665.
\bibitem{Rychkov00}
V.~S. Rychkov, \emph{Linear extension operators for restrictions of function spaces to irregular open sets},
Studia Math. {\bf 140} (2) (2000), 141--162.
\bibitem{Seeger89}
A. Seeger, \emph{A Note on Triebel–Lizorkin Spaces}, In: Approximations and Function Spaces,
Banach Centre Publ. {\bf 22} PWN Polish Sci. Publ., Warszaw,  1989, pp. 391--400.
\bibitem{Seeley72}
R. Seeley, \emph{Interpolation in $L^p$ with Boundary Conditions}, Studia Math. {\bf 44} (1972), 47--60.
\bibitem{Sobolev63}
S.~L. Sobolev, \emph{On a boundary value problem for polyharmonic equations}, Amer. Math. Soc.
Translations (2) {\bf 33} (1963), 1--40.
\bibitem{Stein70}
E. M. Stein, \emph{Singular integrals and differentiability properties of functions},
Princeton Univ. Press, Princeton, 1970.
\bibitem{Strichartz67}
R.~S. Strichartz, \emph{Multipliers on fractional Sobolev spaces}, J. Math.  Mech. {\bf 16} (9) (1967), 1031--1060.
\bibitem{Triebel78}
H. Triebel, \emph{Interpolation Theory, Function Spaces, Differential Operators}, VEB
Deutscher Verlag der Wissenschaften, Berlin, 1978.
\bibitem{Triebel97}
H. Triebel, \emph{Fractals and Spectra}, Birkh\"auser, Basel, 1997.
\bibitem{Triebel01}
H. Triebel, \emph{The Structure of Functions}, Birkh\"auser, Basel, 2001.
\bibitem{Wallin91}
H. Wallin, \emph{The trace to the boundary of Sobolev spaces on a snowflake},
Manuscripta Math. {\bf 73} (1991), 117-125.
\bibitem{Wiener}
N. Wiener, \emph{The Dirichlet problem}, J. Math. and Phys. {\bf 3} (1924), 127--146. Reprinted
in Norbert Wiener: Collected Works with Commentaries, Vol. 1, 394-413, MIT Press,
Cambridge, Massachusetts, 1976.
\bibitem{Zaehle2001}
{M. Z\"ahle, \emph{The average density of self-conformal measures}, J. London Math. Soc. {\bf 63} (3) (2001),
721--734.}
\bibitem{Zaehle}
M. Z\"ahle, \emph{Lectures in Fractal Geometry}, World Scientific, Singapore, 2024.
\bibitem{Ziemer}
W. Ziemer, \emph{Weakly Differentiable Functions}, Graduate Texts in Math. vol. 120, Springer, New York, 1989.
\bibitem{Zhou15}
Y. Zhou, \emph{Fractional Sobolev extension and imbedding}, Trans. Am. Math. Soc. {\bf 367} (2) (2015), 959--979.
\end{thebibliography}
\end{document}